\documentclass[a4paper]{amsart}

\usepackage{amsthm,amsmath,amssymb}
\usepackage{color}
\RequirePackage[numbers]{natbib}
\RequirePackage[colorlinks,citecolor=blue,urlcolor=blue,linkcolor=blue]{hyperref}
\RequirePackage[OT1]{fontenc}

\usepackage[toc,page]{appendix}
\usepackage[utf8]{inputenc}
\usepackage{float}
\usepackage[shortlabels]{enumitem}

\usepackage[textwidth=4cm, textsize=footnotesize]{todonotes}
\numberwithin{equation}{section}
\usepackage{xargs,hyperref,aliascnt,bbm}
\usepackage{ifthen}

\newtheorem{theorem}{Theorem}

\newaliascnt{proposition}{theorem}
\newtheorem{proposition}[proposition]{Proposition}
\aliascntresetthe{proposition}

\newaliascnt{lemma}{theorem}
\newtheorem{lemma}[lemma]{Lemma}
\aliascntresetthe{lemma}

\newaliascnt{corollary}{theorem}
\newtheorem{corollary}[corollary]{Corollary}
\aliascntresetthe{corollary}

\theoremstyle{definition}

\newtheorem{remark}{Remark}

\newaliascnt{definition}{theorem}
\newtheorem{definition}[definition]{Definition}
\aliascntresetthe{definition}

\newcounter{hypA'}

\newenvironment{hyp}[1]{
\begin{enumerate}[label=(\textbf{\sf #1}-\arabic*),resume=hyp#1]\begin{sf}}
{\end{sf}\end{enumerate}}

\newenvironment{addonehypprim}[1]{
\begin{enumerate}[label=(\textbf{\sf #1'}-\arabic*),resume=hyp#1]\begin{sf}\addtocounter{enumi}{-1}}
  {\end{sf}\end{enumerate}}

\def\rme{\mathrm{e}}

\def\Xset{\mathsf{X}}
\def\Yset{\mathsf{Y}}

\def\Zset{\mathsf{Z}}

\def\Xsigma{\mathcal{X}}
\def\Ysigma{\mathcal{Y}}
\def\Zsigma{\mathcal{Z}}

\def\met{\Delta}
\def\rset{\ensuremath{\mathbb{R}}}
\def\zset{\ensuremath{\mathbb{Z}}}
\def\cset{\ensuremath{\mathbb{C}}}

\def\eqsp{}	

\def\PP{\mathbb{P}}

\def\rmd{\mathrm{d}}
\def\rmi{\mathrm{i}}
\newcommandx\LogInt[5][1=\theta,4=,5=Y]{\upsilon_{#4}^{#1}\langle {#5}_{#2:#3} \rangle}

\newcommand{\as}{\mbox{-a.s.}}

\newcommandx{\aslim}[1]{\ensuremath{\stackrel{#1-\text{a.s.}}{\longrightarrow}}}

\newcommandx\sequence[3][2=n,3=\zset]{\ensuremath{\{ #1_{#2}\,:\, #2 \in #3\}}}
\newcommand\nsequence[2]{\ensuremath{\{ #1\,:\, #2\}}}
\newcommandx\dsequence[4][3=n,4=\zset]{\ensuremath{\{ (#1_{#3},#2_{#3})\,:\,#3 \in #4\}}}

\newcommand{\CPE}[3][]
{\ifthenelse{\equal{#1}{}}%
{\mathbb{E}\left[\left. #2 \, \right| #3 \right]}
{\mathbb{E}_{#1}\left[\left. #2 \, \right| #3 \right]}
}
\newcommand{\CPEu}[4][]
{\ifthenelse{\equal{#1}{}}%
{\mathbb{E}\left[\left. #2 \, \right| #3 \right]}
{\mathbb{E}^{#1}_{#2}\left[\left. #3 \, \right| #4 \right]}
}

\newcommand{\CPEv}[3][]
{\ifthenelse{\equal{#1}{}}%
{\mathbb{E}_\star\left[\left. #2 \, \right| #3 \right]}
{\mathbb{E}_\star_{#1}\left[\left. #2 \, \right| #3 \right]}
}

\newcommand{\CPP}[3][]
{\ifthenelse{\equal{#1}{}}%
{\mathbb{P}\left[\left. #2 \, \right| #3 \right]}
{\mathbb{P}_{#1}\left[\left. #2 \, \right| #3 \right]}
}

\newcommand{\CPPu}[3][]
{\ifthenelse{\equal{#1}{}}%
{\mathbb{P}\left[\left. #2 \, \right| #3 \right]}
{\mathbb{P}^{#1}\left[\left. #2 \, \right| #3 \right]}
}

\newcommand{\tCPPu}[3][]
{\ifthenelse{\equal{#1}{}}%
{\tilde{\mathbb{P}}\left[\left. #2 \, \right| #3 \right]}
{\tilde{\mathbb{P}}^{#1}\left[\left. #2 \, \right| #3 \right]}
}

\newcommandx{\chunk}[3]%
{\ensuremath{#1}_{#2:#3}}

\def\1{\mathbbm{1}}

\newcommandx\proj[2][1=,2=]{
\ifthenelse{\equal{#1}{}}
{\operatorname{X}}
{\operatorname{X}_{#1:#2}}i
}

\newcommand{\eqdef}{:=}
\newcommand{\set}[2]{\left\{#1:#2\right\}}
\newcommand{\setvect}[2]{\left(#1\right)_{#2}}

\newcommandx\lkdM[3][1=,3=]{
\ifthenelse{\equal{#2}{}}
{ \mathsf{L}_{#1}^{#3}}
{ \mathsf{L}_{#1}^{#3}\langle #2\rangle}
}
\newcommandx\lkdMStat[3][1=,3=]{
\ifthenelse{\equal{#2}{}}
{ \bar{\mathsf{L}}_{#1}^{#3}}
{ \bar{\mathsf{L}}_{#1}^{#3}\langle #2 \rangle}
}

\newcommandx\lkd[3][1=,3=]{
\ifthenelse{\equal{#2}{}}
{ \ell_{#1}^{#3}}
{ \ell_{#1}^{#3}\langle #2\rangle}
}
\newcommandx\lkdStat[3][1=,3=]{
\ifthenelse{\equal{#2}{}}
{ \bar \ell_{#1}^{#3}}
{ \bar \ell_{#1}^{#3}\langle #2 \rangle}
}

\def\initmle{z^{(\text{\tiny{i}})}}

\def\initmlex{x^{(\text{\tiny{i}})}}

\def\initmleu{u^{(\text{\tiny{i}})}}
\newcommandx\phimomentinit[1][1=\theta]{\phi^{#1}}

\newcommandx{\normLip}[2][1=]{\mathrm{Lip}(#2;#1)}
\newcommandx{\wass}[2][1]{\lVert #2\rVert_{#1}}

\newcommandx{\wasser}[3][1=]{\mathcal{W}_{#1}\left(#2,#3\right)}
\newcommandx{\proho}[3][1=]{\mathcal{P}_{#1}\left(#2,#3\right)}

\newcommandx{\dobru}[2][1=]{\dobrush_{#1}\left( #2\right)}
\newcommand{\dobrush}{\Delta}

\newcommand{\Uset}{\mathrm U}
\newcommand{\Usigma}{\mathcal U}

\newcommandx{\Pcan}[2][1=,2=]{\mathbb{P}_{#1}^{#2}}
\newcommandx{\Ecan}[2][1=,2=]{\mathbb{E}_{#1}^{#2}}

\newcommandx\cesp[4][1=,2=]{\ensuremath{{\mathbb E}_{#1}^{#2}\left[ \left. #3 \right| #4 \right]}}

\newcommandx{\f}[3][1=\theta, 3=]{\psi^{#1}_{#3}\langle #2 \rangle}
\newcommandx{\tf}[2][1=\theta]{\tilde{\psi}^{#1}\langle #2 \rangle}
\newcommandx{\F}[2][1=\theta]{\Psi^{#1}\langle #2 \rangle}
\newcommandx{\tF}[2][1=\theta]{\tilde{\Psi}^{#1}\langle #2 \rangle}
\newcommandx{\flp}[2][1=\theta]{\widehat{\psi}^{#1}\langle #2 \rangle}
\newcommandx{\FLP}[2][1=\theta]{\widehat{\Psi}^{#1}\langle #2 \rangle}
\newcommand{\thv}{{\theta_\star}}
\newcommandx{\kap}[3][1=\theta]{
\ifthenelse{\equal{#3}{}}
{\kappa^{#1}\langle #2\rangle }
{\kappa^{#1}\langle #2\rangle (#3)}
}

\def\zsetp{\zset_{\geq 0}}
\def\zsetn{\zset_{\leq 0}}
\def\zsetpnz{\zset_{>0}}

\def\rsetp{\rset_{\geq0}}

\def\rsetpnz{\rset_{>0}}

\newcommandx{\probdoeblin}[3][1=]{\mu^{#1}_{#2}\langle #3 \rangle}
\newcommand{\Pblock}[2][]
{\ifthenelse{\equal{#1}{}}{\boldsymbol{\operatorname{L}}\langle#2\rangle}{\boldsymbol{\operatorname{L}}^{#1}\langle#2\rangle}
}

\newcommand{\ConPblock}[3][]
{\ifthenelse{\equal{#1}{}}{\boldsymbol{\operatorname{L}}\langle#2|#3\rangle}{\boldsymbol{\operatorname{L}}^{#1}\langle#2|#3\rangle}
}

\newcommand{\pblock}[2][]
{\ifthenelse{\equal{#1}{}}{\mathbf{\ell}\langle#2\rangle}{\mathbf{\ell}^{#1}\langle #2\rangle}
}

\newcommand{\Vset}{\mathsf{V}}
\newcommand{\Vsigma}{\mathcal{V}}

\newcommandx{\limlike}[4][1=\theta, 2=\theta_\star]{p^{#1,#2}\left( #3 | #4 \right)}

\def\Xmet{\boldsymbol{\delta}_\Xset}

\def\Zmet{\boldsymbol{\delta}_\Zset}

\def\Umet{\boldsymbol{\delta}_\Uset}
\newcommand{\Projarg}[2]{\Pi_{#1}\left(#2\right)}
\newcommand{\Proj}[1]{\Pi_{#1}}

\newcommandx{\vnorm}[2][1=V]{\left|#2\right|_{#1}}

\begin{document}

\title[Identifiability of observation-driven models]{Necessary and sufficient conditions for the identifiability of
    observation-driven models}

\author{Randal Douc}
\author{Fran\c{c}ois Roueff}
\author{Tepmony Sim}

 \address{D\'epartement  CITI \\
 CNRS UMR 5157 \\
 T\'el\'ecom SudParis \\
 91000 \'Evry \\
 France}
\email{randal.douc@telecom-sudparis.eu}

\address{LTCI\\
 T\'el\'ecom Paris\\
 Institut Polytechnique de Paris \\
 19 place Marguerite Perey,\\
 91120 Palaiseau \\
 France}
\email{roueff@telecom-paristech.fr}

\address{Department of Foundation Year \\
Institute of Technology of Cambodia \\
 12156 Phnom Penh \\
 Cambodia}
\email{tepmony.sim@itc.edu.kh}

\date{\today}

\begin{abstract}
  \sloppy In this contribution we are interested in proving that a
  given observation-driven model is identifiable. In the case of a
  GARCH$(p,q)$ model, a simple sufficient condition has been
  established in \cite{berkes03} for showing the consistency of the
  quasi-maximum likelihood estimator. It turns out that this condition
  applies for a much larger class of observation-driven models, that
  we call the class of linearly observation-driven models. This class
  includes standard integer valued observation-driven time series such
  as the Poisson autoregression model and its numerous extensions. Our
  results also apply to vector-valued time series such as the bivariate
  integer valued GARCH model, to non-linear models such as the threshold
  Poisson autoregression or to observation-driven models with
  exogenous covariates such as the PARX model.
\end{abstract}

\subjclass[2000]{Primary: 60J05, 62F12; Secondary: 62M05,62M10.}
\keywords{identifiability, observation-driven models, time series of counts}

\maketitle

\section{Introduction}
\sloppy

Observation-driven models (ODM) were introduced in \cite{cox:1981} and
have received considerable attention since. They are commonly used for
modeling various non-linear times series in applications ranging from
economics (see \cite{pindyck1998econometric}), environmental study
(see \cite{bhaskaran2013time}), epidemiology and public health study
(see \cite{zeger1988regression, davis1999modeling,
  ferland:latour:oraichi:2006}), finance (see
\cite{liesenfeld2003univariate,
  rydberg2003dynamics,fokianos:tjostheim:2011,francq2011garch}) and
population dynamics (see \cite{ives2003estimating}). Additional
covariates have been added to some of these models leading to GARCHX
type models, see \cite{AGOSTO2016640} for recent examples in the
context of count data, and the references therein. We include such a
case in our setting leading to the general observation-driven models
with exogenous variables (ODMX).

As often for non-linear time series the question of identifiability of
the observation-driven models is a delicate one and is often appearing
as an assumption used for proving the consistency (say) of the maximum
likelihood estimator. A noticeable exception is the GARCH$(p,q)$
model, for which an explicit sufficient condition appears in
\cite{berkes03}, see their condition (2.27). We will in fact prove
that this condition is not only sufficient but also necessary for the
identifiability, and that this result extends to a much larger class of
observation-driven models than the GARCH($p,q$) model. See
\autoref{thm:gener-ident-cond} below and the comments following this
result.

We provide general conditions to ensure that an ODM or an ODMX
 defined through a collection of
parameterized iterative schemes uniquely describes the law of the
observations. In other words our conditions ensure that two different
iterative schemes within the same model cannot produce the same law
for the observations. Then a given parameter is identifiable if two
different values of the parameter are not compatible with the same
iterative scheme. Let us stress, however, that we do not consider the misspecified
case here, that is, we always assume that the observations indeed
follow the (unique) stationary distribution corresponding to (at
least) one given parameter of the model. Our setting is nevertheless
of interest for the misspecified setting since a non-identifiable
parameter (in the well specified case) cannot be identified in the
misspecified case. Hence the necessity of our conditions remains true
for the misspecifed setting.

A special class of ODMs, that we call \emph{linearly observation
  driven models} (LODMs) below, arises when the hidden variable is
obtained linearly from hidden or observed variables of the past, and
when all these variables are univariate, as for the GARCH($p,q$)
model. This latter model was extensively studied, see for example
\cite{bougerol:picard:1992, francq2004maximum, francq2009tour,
  lindner2009stationarity, francq2011garch} and the references
therein. Many other examples, linear or non-linear, univariate or
multivariate, have been derived from this class, see
\cite{bollerslev08-glossary} for a long list of them, although this
list have been lengthened quite significantly since, in particular
because of the recent adding of various integer valued ODMs to deal
with count time series (see \cite{christou2015count,silva19} and the
references therein). Our goal is to derive necessary and sufficient
conditions potentially applying to a wide variety of ergodic
observation driven models. To illustrate the generality of our
results, we apply them to a list of various examples which includes,
in addition to the standard GARCH model, the nonlinear GARCH model of
\cite{hamadeh-zakoian:jspi2011}, the INGARCH model of
\cite{ferland:latour:oraichi:2006}, the Log-linear Poisson GARCH of
\cite{fokianos:tjostheim:2011}, the MPINGARCH model of
\cite{silva19}, the PARX model of \cite{AGOSTO2016640}, the
Bi-variate integer GARCH model of \cite{cui-zhu-2018} and the
self-excited threshold Poisson Autoregression of \cite{wang2014self}.
We are able to derive necessary and sufficient conditions for
identifiability for all the considered examples.

The rest of the paper is organized as follows.
\autoref{sec:preliminaries} contains additional notation and definitions that
will be used throughout the paper. \autoref{sec:examples} contains
a list of examples already considered in the literature. Our main results can be found
in \autoref{sec:main-results}, some proofs of which are postponed to
\autoref{sec:append:proofs:gen:od}. Before that, in
\autoref{sec:application}, we show how our results apply to the
examples of \autoref{sec:examples} or can be extended to larger
classes of models.


\section{Preliminaries}
\label{sec:preliminaries}
\subsection{Formal definitions of observation driven models}
Let us now formally introduce the class of observation-driven models
and important sub-classes.  Throughout the paper we use the notation
$\chunk{u}{\ell}{m}\eqdef(u_\ell,\ldots,u_m)$ for $\ell \leq m$, with
the convention that $\chunk{u}{\ell}{m}$ is the empty sequence if
$\ell>m$, so that, for instance $(\chunk x0{(-1)},y)=y$.  The
observation-driven time series model can formally be defined as
follows.

\begin{definition}[ODM, ODMX]\label{def:obs-driv-gen}
  Let $(\Xset,\Xsigma)$, $(\Yset,\Ysigma)$ and $(\Uset,\Usigma)$ be measurable
  spaces, respectively called the \emph{latent space}, the \emph{observation
    space} and the \emph{admissible observation space}. Let $(\Theta,\met)$ be a
  compact metric space, called the \emph{parameter space}.  Let $\Upsilon$ be a
  measurable function from $(\Yset,\Ysigma)$ to $(\Uset,\Usigma)$.  Let
  $\set{(\chunk{x}{1}{p},\chunk{u}{1}{q}) \mapsto
    \tilde\psi^\theta_{\chunk{u}{1}{q}}(\chunk{x}{1}{p})}{\theta \in \Theta}$ be a
  family of measurable functions from
  $(\Xset^p\times \Uset^q, \Xsigma^{\otimes p} \otimes \Usigma^{\otimes q})$ to
  $(\Xset, \Xsigma)$, called the \emph{reduced link functions} and let
  $\set{G^\theta}{\theta \in \Theta}$ be a family of probability kernels on
  $\Xset\times\Ysigma$, called the \emph{observation kernels}. A time series
  $\nsequence{Y_k}{k \ge -q+1}$ valued in $\Yset$ is said to be distributed
  according to an \emph{observation-driven model of order $(p,q)$} (hereafter,
  ODM$(p, q)$) with reduced link function $\tilde\psi^\theta$,
  admissible mapping $\Upsilon$ and observation kernel $G^\theta$ if there
  exists a process $\nsequence{X_k}{k \ge -p+1}$ on $(\Xset, \Xsigma)$ such
  that for all $k\in\zsetp$,
\begin{equation}\label{eq:def:gen-ob}
\begin{split}
&Y_{k}\mid \mathcal{F}_{k}\sim G^\theta(X_{k};\cdot)\eqsp,\\
&X_{k+1}=\tilde\psi^\theta_{\chunk{U}{(k-q+1)}{k}}(\chunk{X}{(k-p+1)}{k})\eqsp,
\end{split}
\end{equation}
where
$\mathcal{F}_k=\sigma\left(\chunk{X}{(-p+1)}{k},\chunk{Y}{(-q+1)}{(k-1)}\right)$
and $U_j=\Upsilon(Y_j)$ for all $j>-q$.

In the presence of exogenous variables
defined as an $r$-Markov chain
valued in the space $(\Vset,\Vsigma)$ with kernel $H$, the admissible mapping
$\Upsilon$ is defined from $\Yset\times\Vset$ to $\Uset$ and the
iterative equation~(\ref{eq:def:gen-ob}) is replaced by,
  for all $k\in\zsetp$,
\begin{equation}\label{eq:def:gen-ob-X}
\begin{split}
&V_{k}\mid \mathcal{F}_{k}\sim H(\chunk V{(k-r)}{(k-1)};\cdot)\eqsp,\\
&Y_{k}\mid \mathcal{F}_{k}\sim G^\theta(X_{k};\cdot)\eqsp,\\
&X_{k+1}=\tilde\psi^\theta_{\chunk{U}{(k-q+1)}{k}}(\chunk{X}{(k-p+1)}{k})\eqsp,
\end{split}
\end{equation}
where, in this case,
$\mathcal{F}_k=\sigma\left(\chunk{X}{(-p+1)}{k},\chunk{Y}{(-q+1)}{(k-1)},\chunk{V}{((-r)\wedge
    (-q+1))}{(k-1)}\right)$ and $U_j=\Upsilon(Y_j,V_j)$ for all $j>-q$. We then say that the time series
  $\nsequence{Y_k}{k \ge -q+1}$ valued in $\Yset$ is distributed
  according to an \emph{observation-driven model of order $(p,q)$ with
  $r$-order  Markov exogenous variables} $\nsequence{V_k}{k \ge((-r)\wedge
    (-q+1))}$  (hereafter,
  ODMX$(p, q,r)$) with reduced link function $\tilde\psi^\theta$,
  admissible mapping $\Upsilon$, observation kernel $G^\theta$, and
  exogenous Markov kernel $H$.

  The variables $Y_k$ are called the
  \emph{observed variables}, the variables $X_k$ the
  \emph{hidden variables} and  the variables $U_k$ the
  \emph{admissible variables}. In addition, we define the
  \emph{augmented variables}
\begin{equation}
 \label{eq:Zk:def}
Z_k=\left(\chunk X{(k-p+1)}k,\chunk
  U{(k-q+1)}{(k-1)}\right)\in\Zset\;,
\end{equation}
which take values in the \emph{augmented space}
\begin{equation}
  \label{eq:Zset:def}
\Zset  =\Xset^p\times\Uset^{q-1}\quad\text{endowed with the $\sigma$-field $\Zsigma=\Xsigma^{\otimes p}\otimes\Usigma^{\otimes {(q-1)}}$.}
\end{equation}
\end{definition}

\begin{remark}\label{rem:rem-after-def}
  Let us briefly comment on the unusual notion of \emph{admissible
    mapping} which allows us to define the admissible variables
  $U_j=\Upsilon(Y_j)$ of an ODM($p,q$) in \autoref{def:obs-driv-gen}:
    \begin{enumerate}[label=(\arabic*)]
\item For all $k\geq0$, the conditional
distribution of $(Y_k,X_{k+1})$ given $\mathcal{F}_k $ only depends on
$Z_k$ defined by~(\ref{eq:Zk:def}).
\item The time series $\set{U_k}{k>-q}$ is also
an ODM$(p,q)$ with admissible mapping being the identity, link function
$\tilde\psi^\theta$ and observation kernel
$\tilde G^\theta(x,\cdot)= G^\theta(x,\Upsilon^{-1}(\cdot))$ on the observation
space $(\Uset,\Usigma)$.
\item On the other hand, we can also set the admissible mapping to be
  the identity for the ODM $\set{Y_k}{k>-q}$, in which case
  $\Yset\subseteq\Uset$ and the
  reduced link function should be replaced by the \emph{link function}
  defined
  all $(x,\chunk y1q)\in\Xset^p\times\Yset^q$ by
\begin{equation}
 \label{eq:tilde-psi:def}
  \psi^\theta_{\chunk y1q}(x)= \tilde\psi^\theta_{\chunk
    u1q}(x)\quad\text{with}\quad u_k=\Upsilon(y_k)\quad\text{for}\quad
  1\leq k\leq q\;.
\end{equation}
In fact the advantage of using an admissible mapping is precisely to
obtain a reduced link function $\tilde\psi$, more convenient than the
(non-reduced) link function $\psi$. We will focus in the particular
case where $\tilde\psi$ is linear, into which we can cast not all but
many observation driven models, see Section~\ref{sec:examples}
hereafter.
\item An ODMX($p,q,r$) can be cast into an  ODM($p,q$) by defining
   $\tilde Y_k=(Y_k,\chunk V{(k-r+1)}{k})$ and
$\tilde X_k=(X_k,\chunk V{(k-r)}{(k-1)})$ and observing that the
obtained times series $\nsequence{\tilde Y_k}{k\geq-q+1}$ is an  ODM($p,q$)
with hidden variables $\nsequence{\tilde X_k}{k\geq-p+1}$. However for
treating identifiability as is the purpose here, it is more convenient to
keep distinguishing between the ODM and the ODMX
setting.
\item\label{item:rem-item:tildeG} In the following the variables $U_k$ and $Z_k$ will be used
  extensively as they simplify a lot the presentation and the
  reasoning. It is important to note that the definitions of $U_k$, $Z_k$ and
   $\mathcal{F}_k$ are
  not the same in the ODM and the ODMX settings as they involve $V_k$ in
  the later case. In particular the conditional distribution of $U_k$
  given $\mathcal{F}_k$ takes two very different forms in the ODM and
  ODMX cases. They can be respectively expressed by $\tilde{G}^\theta(X_k;\cdot)$
  and $\tilde{G}^\theta((X_k,\chunk V{(k-r)}{(k-1)});\cdot)$  where
  $\tilde{G}^\theta$ is a probability
  kernel on $\Xset\times\Usigma$ and
  $(\Xset\times\Vset^r)\times\Usigma$, resp. For conciseness
  we use the same notation $\tilde{G}^\theta$ for the two
  cases. They are resp. defined by setting, for all $x\in\Xset$,
  $A\in\Usigma$ and $v\in\Vset^r$,
  \begin{align}
    \label{eq:def:tildeG}
    \tilde{G}^\theta(x,A)&=G^\theta(x,\Upsilon^{-1}(A))\;,\\
    \label{eq:def:tildeG:X}
    \tilde{G}^\theta((x,v),A)&=\int_{\Yset\times\Vset}
                                        \1_A(\Upsilon(y,w))\,G^\theta(x;\rmd
                                        y)\,H(v;\rmd w)\;.
  \end{align}
\end{enumerate}
\end{remark}
When the reduced link function is linear we specify \autoref{def:obs-driv-gen}
into the following.
\begin{definition}[(V)LODM(X)]\label{def:obs-driv-gen:vlodm}
  We say that an ODM($p,q$) (resp. ODMX($p,q,r$)) is a \emph{vector
    linearly observation-driven model of order $(p,q,p',q')$},
  shortened as VLODM$(p,q,p',q')$, (resp. VLODMX$(p,q,r,p',q')$) if
  for some $p',q' \in \zsetpnz$, $\Xset$ and $\Uset$ are closed
  subsets of $\rset^{p'}$ and $\rset^{q'}$, respectively, and, for all
  $x=\chunk{x}{0}{(p-1)}\in\Xset^p$,
  $u=\chunk{u}{0}{(q-1)}\in\Uset^q$, and $\theta\in\Theta$,
\begin{equation}
  \label{eq:tilde-psi:affine:vector:case}
\tilde\psi^\theta_{u}(x)=\boldsymbol{\omega}(\theta)+\sum_{i=1}^p A_i(\theta) \, x_{p-i}
+\sum_{i=1}^q B_i(\theta) \, u_{q-i}\;,
\end{equation}
for some mappings $\boldsymbol{\omega}$, $\chunk A1p$ and $\chunk B1q$
defined on $\Theta$ and valued in $\rset^{p'}$,
$\left(\rset^{p'\times p'}\right)^p$ and
$\left(\rset^{p'\times q'}\right)^q$.  In the case where $p'=q'=1$,
the VLODM$(p,q,p',q')$ (resp.  VLODMX$(p,q,r,p',q')$) is simply called a \emph{linearly
  observation-driven model of order $(p,q)$}, shortened as
LODM$(p,q)$ (resp. LODMX$(p,q,r)$).
\end{definition}
   \subsection{Iterations of the link function}
\label{sec:iter-link-funct}

We now introduce iterated versions of the reduced link function
$\tilde\psi^\theta$.  Let $\Zset$ be defined
by~(\ref{eq:Zset:def}). We define for any $k\in\zsetpnz$ and
$\chunk{u}{0}{(k-1)}\in\Uset^{k}$, the mapping
$\tf{\chunk{u}{0}{(k-1)}}:\Zset\to\Xset$ through a set of recursive
equations of order $(p,q)$.  Namely, for all $n\in\zsetpnz$,
$\chunk{u}{0}{(k-1)}\in\Uset^{k}$ and $z=\chunk z1{(p+q-1)}\in\Zset$,
we define
\begin{align}
  \label{eq:lkd:psi:n:def:rec:u}
&  \tf{\chunk{u}0{(k-1)}}(z)\eqdef x_{k} \;,
\end{align}
where the sequence $\chunk x{(-p+1)}{k}$ is defined by
\begin{align}
  \label{eq:pq-order-rec-equation}
\begin{cases}
u_j=z_{p+q+j} \;,&
    -q< j\leq -1\;, \\
x_j=z_{p+j} \;,&
    -p< j\leq 0\;,\\
x_j=\tilde\psi_{\chunk{u}{(j-q)}{(j-1)}}^\theta\left(\chunk x{(j-p)}{(j-1)}\right)\;,&
    1\leq j\leq k\;.
  \end{cases}
\end{align}
In this set of equations the last line is applied recursively so that in fact,
for all $j\geq1$, $x_{j}$ only depends on $z$ and $\chunk
u0{(j-1)}$.

The equations in~(\ref{eq:pq-order-rec-equation}) define a system with
input sequence $\chunk{u}{(-q+1)}{(k-1)}$, initial condition
$\chunk x{(-p+1)}0$ and output sequence $\chunk x1k$.  Because the
recursion given by the last line of~(\ref{eq:pq-order-rec-equation})
involves $p+1$ successive entries of the output and $q$ successive
entries of the input, it is useful to define blocks, valued in
$\Zset=\Xset^p\times\Uset^{q-1}$ and consider the same recursion
applying to such blocks, hence computing $z_j$ from $z_{j-1}$ and
$u_{j-1}$. Formally, for all $u\in\Uset$, we define
$\tilde\Psi^\theta_u\,:\,\Zset\to\Zset$ by
\begin{align}\label{eq:def:tpsi:gen:od}
  \tilde{\Psi}^\theta_{u}:(\chunk x1p,\chunk u1{(q-1)})
  \mapsto
  \begin{cases}
  \left(\chunk x2p,
  \tilde{\psi}^\theta_{(\chunk u{1}{(q-1)},u)}(\chunk x1p), \chunk u{2}{(q-1)},
  u\right)&\text{ if $q>1$}\\
  \left(\chunk x2p,
    \tilde{\psi}^\theta_{u}(\chunk x1p)\right)&\text{ if $q=1$}\;,
  \end{cases}
\end{align}
\begin{remark}\label{rem:def:gen-ob-Z}
Note in particular that with this notation at hand, and using the
admissible variables $U_k=\Upsilon(Y_k)$ for the ODM case or
$U_k=\Upsilon(Y_k,V_k)$ for the ODMX case, and $Z_k$ defined by~(\ref{eq:Zk:def}), the second line
of~(\ref{eq:def:gen-ob}) and the third line of~(\ref{eq:def:gen-ob-X})
are equivalent to
\begin{equation}\label{eq:def:gen-ob-Z}
Z_{k+1}=  \tilde{\Psi}^\theta_{U_k}(Z_{k})\eqsp.
\end{equation}  
\end{remark}
We further denote the successive composition of  $\tilde\Psi^\theta_{u_0}$,
$\tilde\Psi^\theta_{u_1}$, ..., and $\tilde\Psi^\theta_{u_{k-1}}$ by
\begin{equation}
\label{eq:notationItere:f:cl:gen}
\tF{\chunk{u}{0}{(k-1)}}=\tilde\Psi^\theta_{u_{k-1}} \circ \tilde\Psi^\theta_{u_{k-2}} \circ \dots \circ \tilde\Psi^\theta_{u_0}\,.
\end{equation}
This recursion is the same as the one for defining $\tf{u}$, except
that it is valued in $\Zset$, where as $\tf{u}$ is valued in
$\Xset$. More precisely, denoting, throughout the paper, for all
$j\in\{1,\ldots,p+q-1\}$, by $\Projarg{j}{z}$ the $j$-th entry
of $z\in\Zset$, we have the following
relations between $\tf{u}$ and $\tF{u}$, for all $k\in\zsetp$ and
$u\in\Uset^k$,
\begin{align}
\label{eq:psi:Psi:relation}
\tf{u}&=\Proj p\circ\tF{u} \;,\\
\label{eq:psi:Psi:relation:recip}
  \tF{\chunk{u}{0}{(k-1)}}(z)&=\left(\setvect{\tf{\chunk{u}{0}{j}}(z)}{k-p\leq
                               j<k},\chunk
  u{(k-q+1)}{(k-1)}\right)\;,
\end{align}
where, in the second line, we set $u_j=\Projarg {p+q+j}z$ for $-q< j\leq -1$ and use the convention
$\tf{\chunk{u}{0}{j}}(z)=\Projarg {p-j}z$ for $-p< j\leq 0$.

\subsection{Ergodic assumption and some interesting class of parameters}
\label{sec:ergodic-ass}
In this contribution, we only consider the case where
all processes in the model are ergodic. Namely, we use the following assumption.
\begin{hyp}{A}
\item\label{assum:gen:identif:unique:pi:gen} For all
  $\theta\in\Theta$, there exists a unique stationary solution $\nsequence{(X_k,Y_k)}{k\in\zset}$
satisfying~(\ref{eq:def:gen-ob}).
\end{hyp}
In the case of exogenous covariates this assumption is replaced by
the following.
\begin{addonehypprim}{A}
  \item\label{assum:gen:identif:unique:pi:gen:X}
  For all $\theta\in\Theta$, there exists a unique stationary solution
  $\nsequence{(X_k,Y_k,V_k)}{k\in\zset}$
  satisfying~(\ref{eq:def:gen-ob-X}).
\end{addonehypprim}

This ergodic property is the cornerstone for making statistical inference
theory work and we provide simple general conditions in \cite{douc2015handy}
for $p=q=1$ and in \cite[Chapter~5]{sim-tel-01458087,douc-roueff-sim19} for the
case of general order $(p,q)$.

We now introduce the notation that will allow
us to refer to the stationary distribution of the model throughout the {paper}.
\begin{definition}[Stationary distributions $\PP^\theta$ and
  $\tilde{\mathbb{P}}^\theta$] \label{def:ergo:theta:gen:od}
We define the
  distributions $\PP^\theta$ and $\tilde{\mathbb{P}}^\theta$ as follows.
\begin{enumerate}[label=\alph*)]
\item Under \ref{assum:gen:identif:unique:pi:gen}, $\PP^\theta$ denotes the distribution on
  $((\Xset\times\Yset)^{\zset},(\Xsigma\times\Ysigma)^{\otimes\zset})$ of the stationary solution
  of~(\ref{eq:def:gen-ob});  Under \ref{assum:gen:identif:unique:pi:gen:X}, $\PP^\theta$ denotes the distribution on
  $((\Xset\times\Yset\times\Vset)^{\zset},(\Xsigma\otimes\Ysigma\otimes\Vsigma)^{\otimes\zset})$ of the stationary solution
  of~(\ref{eq:def:gen-ob-X}).
\item Under \ref{assum:gen:identif:unique:pi:gen},
  $\tilde{\mathbb{P}}^\theta$ denotes the projection of  $\PP^\theta$
  on the component $\Yset^{\zset}$;  Under \ref{assum:gen:identif:unique:pi:gen:X},
  $\tilde{\mathbb{P}}^\theta$ denotes the projection of  $\PP^\theta$
  on the component $(\Yset\times\Vset)^{\zset}$.
\end{enumerate}
We also use the symbols
$\mathbb{E}^\theta$ and $\tilde{\mathbb{E}}^\theta$ to denote the expectations
corresponding to $\PP^\theta$ and $\tilde{\mathbb{P}}^\theta$, respectively.
\end{definition}

To study the identifiability of ergodic ODM's, we introduce equivalent
classes that define a partition of the parameter set in subsets of
parameters which share the same distribution of
observations. Formally, it reads as follows.

\begin{definition}[Equivalent classes for
  $\tilde{\mathbb{P}}^\theta$] \label{def:equi:theta:gen:od} Suppose
  that \ref{assum:gen:identif:unique:pi:gen} or  \ref{assum:gen:identif:unique:pi:gen:X} holds and define
  $\tilde{\mathbb{P}}^\theta$ as in \autoref{def:ergo:theta:gen:od}. For all $\theta, \theta'\in\Theta$, we write
  $\theta\sim\theta'$ if and only if
  $\tilde{\mathbb{P}}^{\theta}=\tilde{\mathbb{P}}^{\theta'}$. This
  defines an equivalence relation on the parameter set $\Theta$ and,
  for any $\theta\in\Theta$,
  the equivalence class of $\theta$ is denoted by
  $[\theta]\eqdef\{\theta' \in\Theta:\; \theta'\sim\theta\}$.
\end{definition}
\begin{remark}
  In the context of exogenous variables, that is, under
  \ref{assum:gen:identif:unique:pi:gen:X}, since the distribution of
  $\nsequence{V_k}{k\in\zset}$ under $\tilde{\PP}^\theta$ does not
  depend on $\theta$,
  $\tilde{\mathbb{P}}^{\theta}=\tilde{\mathbb{P}}^{\theta'}$ is
  equivalent to say that the conditional distribution of
  $\nsequence{Y_k}{k\in\zset}$ given $\nsequence{V_k}{k\in\zset}$ is
  the same under $\tilde{\PP}^\theta$ and under
  $\tilde{\PP}^{\theta'}$.
\end{remark}
Determining the equivalent classes $[\theta]$ for all
$\theta\in\Theta$ amounts to solve the identifiability of a parameter
under the assumption of a \emph{well specified} model. Namely,
assuming that the distribution of the observations is given by
$\tilde\PP^\thv$ for some (unknown) parameter $\thv\in\Theta$, a
parameter $\xi(\thv)$ is identifiable if and only if the given mapping
$\xi$ is constant over the equivalent class $[\thv]$. Without
identifiability, the consistency of any estimator of $\xi(\thv)$ is
not possible.  A special case is when $[\thv]$ reduces to the
singleton $\{\thv\}$, so that every parameter $\xi(\thv)$ is
identifiable, in which case the model is said to be identifiable.
Obviously, if $\theta$ and $\thv$ share the same iterative
equation~(\ref{eq:def:gen-ob}) (or ~(\ref{eq:def:gen-ob-X} with
exogenous covariates), that is, if $G^{\theta}=G^{\thv}$ and
$\tilde\psi_u^\theta(x)=\tilde\psi_u^\thv(x)$ for all
$(u,x)\in\Uset^q\times\Xset^p$, by uniqueness of the stationary
distribution, they must share the same one and in particular we get
$\tilde{\mathbb{P}}^\theta=\tilde{\mathbb{P}}^\thv$. Thus, using the more convenient notation $\tilde{\Psi}$ introduced
in~(\ref{eq:def:tpsi:gen:od}),
we have
\begin{equation}
  \label{eq:ident_fundamental}
\set{\theta\in\Theta}{G^{\theta}=G^{\thv}\text{ and
    }\tilde{\Psi}^{\theta}_u(z)=\tilde{\Psi}^{\thv}_u(z)\text{ for all
    }(z,u)\in\Zset\times\Uset}\subseteq  [\thv]
  \;.
\end{equation}
We will provide general conditions ensuring that this inclusion
becomes an equality, see \autoref{thm:the-main-result-general-setting}
below. However it may happen in standard situations that this
inclusion is strict, as will be seen in
\autoref{rem:standard-lodms}\ref{item:exple-strict-inclusion}. Nevertheless,
in all the considered examples, it will be possible to recover an equality by
replacing $\Zset$ by a more appropriate subset in the left-hand side
of~(\ref{eq:ident_fundamental}). 

As often for ODMs, our results rely on the assumption that,
under $\PP^\theta$, the hidden variables are measurable with respect
to the admissible variables from the past. This is not completely surprising since,
using the notation introduced in \autoref{sec:iter-link-funct}, iterating the link
function, we have that, for all $\theta\in\Theta$ and all $s<t$ in
$\zset$,
$$
X_t=\tf{\chunk{U}{s}{(t-1)}}(Z_s)\qquad\PP^\theta\as
$$
In particular, taking $t=1$ and letting $s$ decrease backward towards
$-\infty$, we get that, $X_1$ is measurable with respect to
$\cap_{t\in\zset}\left(\mathcal{F}^Z_t\vee\mathcal{F}^U_0\right)$, where
  $(\mathcal{F}^Z_t)$ and $(\mathcal{F}^U_t)$  respectively denote
  the natural filtrations of $\sequence{Z}$ and $\sequence{U}$. To our
  knowledge, all ODM of interest satisfy in fact the stronger property
  that $X_1$  is measurable with respect to $\mathcal{F}^U_0$, which
  is sometimes called the \emph{invertibility condition}. This
  condition is now introduced with some notation for expressing $X_1$
  as a measurable function of $\chunk U{(-\infty)}0$.
  \begin{hyp}{A}
\item \label{item:X:from:Y:determ}
For all
$\theta\in\Theta$, the measurable function
  $\tf[\theta]{\cdot}:\Uset^{\zsetn}\to\Xset$ satisfies
  \begin{equation}
    \label{eq:identif-XY}
X_1 =  \tf[\theta]{\chunk{U}{(-\infty)}{0}} \qquad \PP^\theta\as
  \end{equation}
\end{hyp}
Since $\PP^\theta$ is stationary,~(\ref{eq:identif-XY}) also implies
that, for all $t\in\zset$,
$X_{t+1} = \tf[\theta]{\chunk{U}{(-\infty)}{t}}$ $\PP^\theta\as$
For an ODM (resp. an ODMX), we have $U_k=\Upsilon(Y_k)$ (resp. $U_k=\Upsilon(Y_k,V_k)$).
Thus
Assumption~\ref{item:X:from:Y:determ} allows us to derive the $X_t$'s
from the $Y_t$'s (resp. from the $Y_t$'s and $V_t$'s) and therefore to
rewrite the relationship given through the link function in the second
line of~(\ref{eq:def:gen-ob}) (resp. in the third line
of~(\ref{eq:def:gen-ob-X})) between these variables in terms of a
recursive relationship involving only the $Y_t$'s (resp. the $Y_t$'s
and the $V_t$'s). It turns out that Condition~(\ref{eq:identif-XY}) in
\ref{item:X:from:Y:determ} can be verified using $\tilde\PP^\theta$
only, that is, we do not need $\PP^\theta$ but only its marginal onto
the variable $Y_k$'s (resp. the variables $Y_t$'s and the
$V_t$'s), as shown by the following result.
\begin{lemma}
  \label{lem:identif-XY}
  Consider an ODM$(p,q)$ satisfying
  \ref{assum:gen:identif:unique:pi:gen} with $p,q\in\zsetpnz$ or an
  ODMX$(p,q,r)$ satisfying \ref{assum:gen:identif:unique:pi:gen:X}
  with $p,q,r\in\zsetpnz$.  Let $\theta\in\Theta$ and consider a
  measurable function
  $\tf[\theta]{\cdot}:\Uset^{\zset_-}\to\Xset$. Then~(\ref{eq:identif-XY})
  is satisfied if and only if the two following equations hold.
 \begin{align}
    \label{eq:identif-XYa}
&  \tf[\theta]{\chunk{U}{(-\infty)}{0}}=\tilde \psi^\theta_{\chunk
   U{(-q+1)}{0}}\left(\setvect{\tf[\theta]{\chunk{U}{(-\infty)}{j}}}{-p\leq j\leq-1}\right)
    \qquad \tilde\PP^\theta\as    \\
       \label{eq:identif-XYb}
&       \tCPPu[\theta]{Y_1\in\cdot}{\chunk Y{(-\infty)}0} =G^\theta\left(\tf[\theta]{\chunk{U}{(-\infty)}{0}},\cdot\right)    \quad \tilde\PP^\theta\as
 \end{align}
\end{lemma}
\begin{proof}
Suppose that~(\ref{eq:identif-XY}) holds true. Since
$\PP^\theta$ is shift invariant, it can be extended to all time
instants $k\in\zset$, namely,
$$
X_k =  \tf[\theta]{\chunk{U}{(-\infty)}{(k-1)}} \qquad \PP^\theta\as
$$
But then~(\ref{eq:identif-XYa}) and~(\ref{eq:identif-XYb}) follows from the
model equations~(\ref{eq:def:gen-ob}) in the case of an ODM or
~(\ref{eq:def:gen-ob-X}) in the case of an ODMX.

Suppose now that~(\ref{eq:identif-XYa}) and~(\ref{eq:identif-XYb}) hold
true. Since
$\PP^\theta$ is shift invariant, they are extended  to all time instants
$k\in\zset$  in the form
 \begin{align*}
&  \tf[\theta]{\chunk{U}{(-\infty)}{k-1}}=\tilde \psi^\theta_{\chunk
   U{(k-q)}{(k-1)}}\left(\setvect{\tf[\theta]{\chunk{U}{(-\infty)}{j}}}{k-p-1\leq j\leq k-2}\right)
    \quad \tilde\PP^\theta\as    \\
&       \tCPPu[\theta]{Y_k\in\cdot}{\chunk Y{(-\infty)}{(k-1)}} =G^\theta\left(\tf[\theta]{\chunk{U}{(-\infty)}{(k-1)}},\cdot\right)    \quad \tilde\PP^\theta\as
 \end{align*}
 Defining $X'_k=\tf[\theta]{\chunk{U}{(-\infty)}{(k-1)}}$ for all
 $k\in\zset$, we see that $\nsequence{(X'_k,Y_k)}{k\in\zset}$ is a
 stationary sequence satisfying the model
 equations~(\ref{eq:def:gen-ob}) in the ODM case and
 $\nsequence{(X'_k,Y_k,V_k)}{k\in\zset}$ is a stationary sequence
 satisfying the model equations~(\ref{eq:def:gen-ob-X}) in the ODMX
 case. By uniqueness of $\PP^\theta$ assumed in
 \ref{assum:gen:identif:unique:pi:gen} and
 \ref{assum:gen:identif:unique:pi:gen:X}, respectively, we get
 that~(\ref{eq:identif-XY}) holds.
\end{proof}
Now, given
$\thv\in\Theta$, we  introduce the set  $\langle\thv\rangle$ of all parameters
$\theta\in\Theta$ whose recursive relationship~(\ref{eq:identif-XYa})  apply to
almost all trajectories of $\sequence{U}$ under the distribution of
$\thv$.
\begin{definition}[Subset
  $\langle\thv\rangle$] \label{def:ident:theta:gen:od:link} Suppose
  that we are given a measurable
  function $\tf[\theta]{\cdot}:\Uset^{\zsetn}\to\Xset$. Then, for all
  $\thv\in\Theta$, we denote by $\langle\thv\rangle$ the set of all
  parameters $\theta\in\Theta$ satisfying the two following equations
  \begin{align}
    \label{eq:to-prove-deidentif:gen:od}
&    \tf[\theta]{\chunk{U}{(-\infty)}{0}}=\tf[\thv]{\chunk{U}{(-\infty)}{0}} &
    \tilde{\PP}^\thv\as\;,\\
    \label{eq:identif-XYa-thv}
&  \tf[\theta]{\chunk{U}{(-\infty)}{0}}=\tilde \psi^\theta_{\chunk
   U{(-q+1)}{0}}\left(\setvect{\tf[\theta]{\chunk{U}{(-\infty)}{j}}}{-p\leq j\leq-1}\right)
     &\tilde\PP^\thv\as
  \end{align}
\end{definition}
It is important to note that  $\langle\thv\rangle$ of \autoref{def:ident:theta:gen:od:link}
depends on the choice of the class of functions
$\set{\tf[\theta]{\cdot}}{\theta\in\Theta}$ and that Assumption~\ref{item:X:from:Y:determ} alone is not sufficient to
define each $\tf[\theta]{\cdot}$ on the whole set $\Uset^{\zsetn}$ of
trajectories, since Relation~(\ref{eq:identif-XY}) is only
required to hold $\PP^\theta\as$.
We now provide  some Lipschitz condition on the iterates of the link
function $\tilde \psi^\theta$ and a moment condition on $U_0$ that allow us
to build a natural class of functions
$\set{\tf[\theta]{\cdot}}{\theta\in\Theta}$ that
satisfies~\ref{item:X:from:Y:determ}.
Whenever we need some metric on the space $\Zset$, we assume the following.
\begin{hyp}{A}
\item\label{item:CCMShyp} The  $\sigma$-fields $\Xsigma$ and $\Usigma$
are Borel ones, respectively associated to $(\Xset,\Xmet)$ and
$(\Uset,\Umet)$, both assumed to be complete and separable metric spaces.
\end{hyp}
Recall that, for any finite $\Uset$-valued sequence $u$, the mapping $\tf{u}$ is
defined by~(\ref{eq:lkd:psi:n:def:rec:u}) following the recursion
in~(\ref{eq:pq-order-rec-equation}). Define, for all $n\in\zsetpnz$, the
Lipschitz constant for $\tf{u}$, uniform over $u\in\Uset^{n}$,
  \begin{equation}
    \label{eq:Lip:constant:n:def}
\mathrm{Lip}_n^\theta=\sup\set{\frac{\Xmet(\tf{u}(z),\tf{u}(z'))}{\Zmet(z,z')}}{(z,z',u)\in\Zset^{2}\times\Uset^{n}}\;,
\end{equation}
where we set, for all $v=\chunk v1{(p+q-1)}\in\Zset$ and $v'=\chunk{v'}1{(p+q-1)}\in\Zset$,
\begin{equation}
  \label{eq:def:Zmet}
  \Zmet(v,v')=\left(\max_{1\leq k\leq p}\Xmet(v_k,v'_k)\right)\,\bigvee\,
  \left(\max_{p< k<p+q}\Umet(v_k,v'_k)\right)\;.
\end{equation}
We use the following assumptions to define the class of functions $\set{\tf[\theta]{\cdot}}{\theta\in\Theta}$.
\begin{hyp}{A}
\item\label{assum:bound:rho:gen}  For all $\theta\in\Theta$, we have $\mathrm{Lip}_1^\theta<\infty$ and
  $\mathrm{Lip}_n^\theta\to0$ as $n\to\infty$.
\item\label{item:ident-cond-log-moment}
  There exists   $\initmlex_1\in\Xset$ and, if $q>1$,
  $\initmleu_1\in\Uset$ such that the constant vectors
  $\initmlex=(\initmlex_1,\dots,\initmlex_1)\in\Xset^p$ and
  $\initmleu=(\initmleu_1,\dots,\initmleu_1)\in\Uset^{q-1}$
  satisfy, for all $\thv, \theta\in\Theta$,
  \begin{equation}
    \label{eq:cond:moment:XY}
\mathbb{E}^\thv\left[\phimomentinit(U_0)\right]<\infty\;,
\end{equation}
where we defined, for all $u\in\Yset$,
$$
\phimomentinit(u)=
\ln^+\left(\Xmet\left(\initmlex_1,\tilde\psi^\theta_{(\initmleu,u)}(\initmlex)\right)\vee\Umet(\initmleu_1,u)\right)
$$
with the convention $\Umet(\initmleu_1,u)=0$ if $q=1$.
\item\label{item:continuity-in-x-for-u-given} For all
  $\theta\in\Theta$ and $u\in\Uset^q$, the reduced link function
  $\tilde\psi^\theta_u$ is continuous on $\Xset^p$.
\end{hyp}

Obviously, under~\ref{assum:bound:rho:gen}, for all
$\theta\in\Theta$ and  $u\in\Uset^{\zsetn}$, the asymptotic
behavior of $\tf{\chunk u{(-n)}0}(z)$ as $n\to\infty$ does not depend
on $z\in\Zset$. We can thus denote
\begin{align}
  \label{eq:definition:f:asalimit:with:domain}
\begin{cases}
\mathrm{D}^\theta&:=\set{u\in\Uset^{\zsetn}}{\tf{\chunk
  u{(-n)}0}(z)\text{ converges in $\Xset$ as }n\to\infty}\\
\displaystyle  \tf[\theta]{u}&:=\lim_{n\to\infty}\tf{\chunk u{(-n)}0}(z)\text{ for
                 all }u\in\mathrm{D}^\theta\;,
             \end{cases}
\end{align}
and keep in mind that the initial point $z$ has no influence on these
two definitions.

By~(\ref{eq:psi:Psi:relation:recip}), we
further have the following result using the
definitions in~(\ref{eq:definition:f:asalimit:with:domain}).
\begin{align}
  \label{eq:definition:fF:asalimit:with:domain}
\begin{cases}
\mathrm{D}^\theta&:=\set{u\in\Uset^{\zsetn}}{\tF{\chunk
  u{(-n)}0}(z)\text{ converges in $\Zset$ as }n\to\infty}\\
\displaystyle  \tF[\theta]{u}&:=\lim_{n\to\infty}\tF{\chunk u{(-n)}0}(z)\text{ for
                 all }u\in\mathrm{D}^\theta\;,
             \end{cases}
\end{align}
and $\tF[\theta]{u}$ and $\tf{u}$ are related for all
$u\in\mathrm{D}^\theta$  through the formulas
\begin{align}
\label{eq:psi:Psi:relation;infty}
\tf{u}&=\Proj p\circ\tF{u} \;,\\
\label{eq:psi:Psi:relation:recip:infty}
  \tF[\theta]{u}&=\left(\left(\tf[\theta]{\chunk
        u{(-\infty)}{k}}\right)_{-p< k\leq0},\chunk u{(-q+2)}0\right)\;.
\end{align}
Based on these definitions, we now introduce
subsets of $\Zset$ of particular interest.
\begin{definition}[Set $\mathrm{E}^\theta$]
  \label{def:E-sets}
If  Assumption~\ref{assum:bound:rho:gen} holds, we set, for any
$\theta\in\Theta$,
\begin{equation}
  \label{eq:E-theta-def}
\mathrm{E}^\theta\eqdef\set{\tF{u}}{u\in\mathrm{D}^\theta}\subset\Zset\;,
\end{equation}
where $\tF[\theta]{\cdot}$ and $\mathrm{D}^\theta$ are defined
by~(\ref{eq:definition:fF:asalimit:with:domain}).
\end{definition}
\begin{remark}\label{rem:E-theta-Zk}
Suppose that, for all
$\theta\in\Theta$, we have $\chunk U{(-\infty)}0\in\mathrm{D}^\theta$,
$\tilde\PP^\theta\as$, and suppose that~\ref{item:X:from:Y:determ}
holds for $\tf[\theta]{\cdot}$ as
in~(\ref{eq:definition:f:asalimit:with:domain}). Then, 
by~(\ref{eq:Zk:def}) and~(\ref{eq:psi:Psi:relation:recip:infty}), we
have $Z_1\in\mathrm{E}^\theta$, $\PP^\theta\as$ Since $\PP^\theta$ is
shift-invariant, we get that $\set{Z_k}{k\in\zset}$ takes its values
in $\mathrm{E}^\theta$,  $\PP^\theta\as$ This is why the set
$\mathrm{E}^\theta$ will be of interest in the following. 
\end{remark}
The following result is proved in \autoref{sec:proof-XY-cond}.
\begin{lemma}
  \label{lem:Lip:XY-cond}
Consider an ODM$(p,q)$ satisfying
\ref{assum:gen:identif:unique:pi:gen} with $p,q\in\zsetpnz$ or an  ODMX$(p,q,r)$ satisfying
\ref{assum:gen:identif:unique:pi:gen:X} with $p,q,r\in\zsetpnz$.
Suppose that \ref{item:CCMShyp},~\ref{assum:bound:rho:gen}
and~\ref{item:ident-cond-log-moment} hold. Then, for all $\theta,\thv\in\Theta$,
we have $\chunk U{(-\infty)}{0}\in\mathrm{D}^\theta$,
$\tilde\PP^{\thv}\as$, \ref{item:X:from:Y:determ} holds and, setting
$\mathrm{E}^\thv$ as in \autoref{def:E-sets}, we have
\begin{equation}
  \label{eq:ident_fundamental_E}
\set{\theta\in\Theta}{G^{\theta}=G^{\thv}\text{ and
    }\tilde{\Psi}^{\theta}_u(z)=\tilde{\Psi}^{\thv}_u(z)\text{ for all
    }(z,u)\in\mathrm{E}^\thv\times\Uset}\subseteq  [\thv]
  \;.
\end{equation}
If
moreover~\ref{item:continuity-in-x-for-u-given} is assumed,
then~(\ref{eq:identif-XYa-thv}) holds for all
$\theta\in\Theta$. Consequently, the set $\langle\thv\rangle$ in
\autoref{def:ident:theta:gen:od:link} can be expressed as
\begin{equation}
    \label{eq:to-prove-deidentif:def-angle-set-limit-case}
\langle\thv\rangle=\set{\theta\in\Theta}{\tf[\theta]{\chunk{U}{(-\infty)}{0}}=\tf[\thv]{\chunk{U}{(-\infty)}{0}} \quad
    \tilde{\PP}^\thv\as}\;.
\end{equation}
\end{lemma}
\begin{remark}
  \label{rem:ergod-assumpt-invert}
  The invertibility Assumption~\ref{item:X:from:Y:determ} is
    essential for deriving the identifiability class $[\thv]$
    using the set $\langle\thv\rangle$. \autoref{lem:Lip:XY-cond} can
    be used to prove it in all the examples that are considered
    hereafter. Indeed as will be checked in \autoref{sec:application},
    all the considered examples satify the following facts:
  \begin{enumerate}[label=(\arabic*)]
  \item The sets $\Xset$ and $\Uset$ are closed
    subsets of finite dimensional normed spaces and \ref{item:CCMShyp}
    follows.
  \item \label{item:lyapunovVSinvert}
    Assumption~\ref{assum:bound:rho:gen} is weaker than what
    is needed for proving the ergodicity
    assumption~\ref{assum:gen:identif:unique:pi:gen}. Consider for
    instance the classical GARCH(1,1) model defined by setting
    $\Upsilon(y)=y^2$, $\tilde\psi^\theta_u(x)=\omega+ax+bu$ and
    $G^\theta(x,\cdot)=\PP(x\varepsilon\in\cdot)$ where $\varepsilon$ is
    centered with variance 1. Then it is easily seen
    that~\ref{assum:bound:rho:gen} is equivalent to $a<1$. On the
    other hand, the Lyapunov condition to
    get~\ref{assum:gen:identif:unique:pi:gen} reads
    $\mathbb{E}\log(b\epsilon^2+a)<0$, which implies $a<1$.
  \item The moment condition \ref{item:ident-cond-log-moment} is
    implied by
    $\mathbb{E}^\thv\left[\log^+(|Y_0|)\right]<\infty$, where
    $|\cdot|$ is some norm, and this condition  holds as
    a byproduct of the proof
    of~\ref{assum:gen:identif:unique:pi:gen} (which often imply
    $\mathbb{E}^\thv\left[|Y_0|^s\right]$ for some $s>0$).
  \item One can readily checks~\ref{item:continuity-in-x-for-u-given}.
  \end{enumerate}
  Note also that the set in the left-hand side
  of~(\ref{eq:ident_fundamental_E}) contains the set in the left-hand side
  of~(\ref{eq:ident_fundamental}). In all our examples, the
  assumptions of~\ref{thm:the-main-result-general-setting} below will be
  shown to hold, implying that the inclusion
  in~(\ref{eq:ident_fundamental_E}) is in fact an equality.
  In some of these examples, however, the inclusion in~(\ref{eq:ident_fundamental})
  is strict, showing that the sets in the left-hand sides
  of~(\ref{eq:ident_fundamental}) and~(\ref{eq:ident_fundamental_E})
  may happen to be different.  
\end{remark}

\section{Examples}\label{sec:examples}
We give a non-exhaustive list of possible examples related to the
previous definitions and for which our results apply, as will be shown
in \autoref{sec:application}.
\subsection{Standard LODMs}\label{sec:standard-lodm}
Many models can be considered as an LODM by choosing an
appropriate admissible mapping $\Upsilon$.

\noindent\textbf{GARCH}.
   The standard GARCH$(p,q)$ model is a special case of
  LODM$(p, q)$, in which case $\Xset=\rsetp$, $\Yset=\rset$,
  $\Upsilon(y)=y^2$, and $G^\theta(x,\cdot)$ is a centered distribution with
  variance $x$, most commonly the normal distribution.

\noindent\textbf{INGARCH}.
 The standard
  Poisson integer-valued GARCH (INGARCH, see
  e.g. \cite{ferland:latour:oraichi:2006}) obviously is an  LODM$(p,
  q)$ with $\Xset=\rsetp$, $\Yset=\zsetp$ and $G^\theta(x,\cdot)$ is
  the Poisson distribution with mean $x$.

\noindent\textbf{Extensions} of INGARCH.
Many extensions of the INGARCH model simply consist in extending
the Poisson distribution to more general ones: the NBIN-GARCH
model of \cite{zhu:2011}, the COM-Poisson INGARCH model of \cite{zhu-com:2012}, the
zero-inflated Poisson GARCH of \cite{zhu-zero:2012}, or the mixed-Poisson
integer GARCH (MPINGARCH) of \cite{silva19}, among
others. Often for these extensions, an extra-parameter is used to
define the distribution $G^\theta(x,\cdot)$, in which case this extra
parameter can be taken either as known, in which case $G^\theta$ does
not depend on $\theta$, or as unknown, in which case $G^\theta$ only depends on a
subparamater of $\theta$.  Some integer valued observation driven models require
using a non-identity admissible mapping in order to be seen as an
LODM. For instance, the log-linear Poisson Garch model of
\cite{fokianos:tjostheim:2011} is an LODM$(p, q)$ by taking
$\Upsilon(y)=\ln(1+y)$, and $G^\theta(x,\cdot)$ as the Poisson
distribution with mean $\rme^x$.

All the above examples are LODMs with a similar parametrization of the
linear link function. In fact they only differ through the admissible
mapping $\Upsilon$ or the observation kernel $G^\theta$. We assemble
them using the following definition.  
\begin{definition}[Standard LODM (with unknown observation kernel)]
  \label{def:standard-lodms} An LODM($p,q$)
  of \autoref{def:obs-driv-gen:vlodm} is said to be \emph{standard} if
  $\theta=(\omega,\chunk a1p,\chunk b1q)\in\Theta\subset\rset^{1+p+q}$
  with $\boldsymbol{\omega}(\theta)=\omega$, $A_k(\theta)=a_k$ for all
  $1\leq k\leq p$ and $B_k(\theta)=b_k$ for all $1\leq k\leq q$, and
  $G^\theta$ does not depend on $\theta$, in which case we denote it
  by $G$. It is said to be standard \emph{with unknown observation kernel} if
  the same holds with
  $\theta=(\vartheta,\varphi)\in\Theta\subset\rset^{1+p+q}\times\Phi$
  where $\vartheta=(\omega,\chunk a1p,\chunk b1q)$ and $\Phi$ is some
  parameter set, and $G^\theta$ only depends on $\varphi$, in which
  case we denote it by $G^{\varphi}$.
\end{definition}
In this definition the parameter $\varphi$ is used in the case where
the observation kernel depends on an unknown extra parameter, as
considered in \cite{silva19} for the class of MPINGARCH($p,q$) models
which include the NBIN GARCH model.  A necessary and sufficient
condition for standard LODMs with known or unknown observation kernel
is provided in \autoref{thm:gener-ident-cond} below and applies to all
the examples listed in this section.

  \subsection{A bivariate example}\label{sec:bivariate-example}
Let us extend \autoref{def:standard-lodms} to the vector case as follows.
\begin{definition}[Standard VLODM (with unknown observation kernel)]
  \label{def:standard-vlodms} A VLODM($p,q,p',q'$) is said to be \emph{standard} if
  $\theta=(\boldsymbol{\omega},\chunk A1p,\chunk
  B1q)\in\Theta\subset\rset^{p'}\times\left(\rset^{p'\times p'}\right)^p\times\left(\rset^{p'\times q'}\right)^q$
  with $\boldsymbol{\omega}(\theta)=\omega$, $A_k(\theta)=A_k$ for all
  $1\leq k\leq p$ and $B_k(\theta)=B_k$ for all $1\leq k\leq q$, and
  $G^\theta$ does not depend on $\theta$, in which case we denote it
  by $G$. It is said to be standard \emph{with unknown observation kernel} if
  the same holds with
  $\theta=(\vartheta,\varphi)\in\Theta\subset \rset^{p'}\times\left(\rset^{p'\times p'}\right)^p\times\left(\rset^{p'\times q'}\right)^q
  \times\Phi$
  where $\vartheta=(\boldsymbol{\omega},\chunk A1p,\chunk B1q)$ and $\Phi$ is some
  parameter set, and $G^\theta$ only depends on $\varphi$, in which
  case we denote it by $G^{\varphi}$.
\end{definition}
Then the bivariate integer valued GARCH model of \cite{cui-zhu-2018}
is a \emph{standard} VLODM($1,1,2,2$) with unknown observation kernel
defined for all $\varphi\in\Phi=[-\overline{\varphi},\overline{\varphi}]$ (where
$\overline{\varphi}>0$ is some constant), $(x_1,x_2)\in\rsetpnz$ and
$(y_1,y_2)\in\zsetp$, by
 \begin{equation}
  \label{eq:bivariate-poisson}
 G^{\varphi}((x_1,x_2),\{y_1,y_2\})=\frac{x_1^{y_1}x_2^{y_2}}{y_1!\,y_2!}\rme^{-(x_1-x_2)}\;\left(1+\varphi\prod_{i=1,2}(\rme^{-y_i}-\rme^{cx_i})\right)\;.
\end{equation}
where $c=1-1/\rme$. Since we have $p=q=1$ in this example, we simply
denote
$\theta=(\boldsymbol{\omega},A,B,\varphi)\subset\rset^2\times\rset^{2\times2}\times\rset^{2\times2}\times[-\overline{\varphi},\overline{\varphi}]$.

\subsection{Non-linear GARCH}\label{sec:non-linear-garch}
The non-linear GARCH model of \cite{hamadeh-zakoian:jspi2011} is an
ODM($p,q$) with
\begin{align*}
  G^\theta(x,\cdot)&= \PP(x^{1/\delta}\eta\in\cdot ) \\
  \psi^\theta_{\chunk{y}{0}{(q-1)}}(\chunk{x}{0}{(p-1)})&= \omega + \sum_{i=1}^{p} a_i x_{p-i}+ \sum_{i=1}^{q} \left(b_{i}(1) (y_{q-i}^+)^\delta + b_{i}(2) (y_{q-i}^-)^\delta \right)  \; ,
\end{align*}
where $\eta$ is a real valued random variable.  Two cases
are considered in  \cite{hamadeh-zakoian:jspi2011} :
\begin{enumerate}[label=Case \arabic*)]
\item\label{item:non-linear-case1} If the exponent $\delta$ is known, we set
$\theta=(\omega,\chunk{a}{1}{p},\chunk{\mathbf{b}}{1}{q})$ with
$\mathbf{b}_k=\begin{bmatrix}b_k(1)&b_k(2)\end{bmatrix}$ for $k=1,\dots,q$,
and
$\Upsilon(y)=( (y^+)^\delta,(y^-)^\delta)$, in which case we have a
standard VLODM$(p,q,1,2)$ of \autoref{def:standard-vlodms} with known observation kernel and with the
parameters $A_k$ denoted by $a_k$ for $k=1,\dots,p$ and the parameters
$B_k$ denoted by
$\mathbf{b}_k$ for $k=1,\dots,q$.
\item\label{item:non-linear-case2} If the exponent $\delta$ is unknown, we set
  $\theta=(\omega,\chunk{a}{1}{p},\chunk{\mathbf{b}}{1}{q},\delta)$
  and $\Upsilon(y)=( y^+,y^-)$, in which case $\delta$ must be
  included in the definition of $\tilde \psi^\theta$.
\end{enumerate}
To our best knowledge this
kind of model have not be extended to the case of (signed) integer
valued time series.

\subsection{The SETPAR  model}\label{sec:other-non-linear}
Other non-linear ODM's that cannot be cast into an LODM or a VLODM can
be found in \cite{christou-fkianos2015count}. We consider here
the \emph{self-excited threshold Poisson autoregression} (SETPAR) model originally studied in
\cite{wang2014self}, which is an ODM(1,1), integer valued
($\Yset=\zsetp$), with link function defined for all
$\theta=(\omega_1,\omega_2,a_1,a_2,b_1,b_2,r)\in\Theta\subset\rsetp^6\times\zsetp$ by
\begin{equation}
  \label{eq:tpar-link}
\psi^\theta_{y}(x)=
\begin{cases}
  \omega_1 +  a_1 x+ b_1 y&\text{ if $y\leq r$}  \\
  \omega_2 +  a_2 x+ b_2 y&\text{ if $y> r$}        \; ,
\end{cases}
\end{equation}
with
$G^\theta(x,\cdot)$ being the usual Poisson distribution with mean
$x$.
\subsection{The PARX model}\label{sec:parx-model}
 Our last example is the Poisson autoregression with exogenous
 covariates (PARX) model of \cite{AGOSTO2016640}. The PARX model is
 similar to the standard INGARCH($p,q$) model above but with
 additional exogenous variables entering into the link function
 for generating the hidden variables. The exogenous variables are assumed to satisfy some Markov dynamic of
 order 1  (see \cite[Assumption~1]{AGOSTO2016640}).
 Thus it is an ODMX($p,q,1$). Eq.~(1) in \cite{AGOSTO2016640}
 corresponds to setting our $G^\theta(x,\cdot)$ as the Poisson
 distribution with mean $x$. Eq.~(2) in
 \cite{AGOSTO2016640} corresponds to setting for all $x=
 \chunk x0{(p-1)}\in\rset^p$ and $u=(\chunk y0{(q-1)},v)\in\rset^q\times\Vset$,
  \begin{equation}
   \label{eq:Parx-link-general}
 \tilde{\psi}_u(x)=\omega+\sum_{k=1}^pa_k x_{p-k}+\sum_{k=1}^qb_k y_{q-k}+f(v,\gamma)\;,
\end{equation}
where $f(\cdot,\gamma):\Vset\to\rsetp$ is a known function and
 $\theta=(\omega,a_1,\dots,a_p,b_1,\dots,b_q,\gamma)$ is the unknown
 parameter of the model. Note that our $Y_k$, $X_k$, $V_k$,
 $\chunk a1p$ and $\chunk b1q$ correspond to their $y_k$, $\lambda_k$,
 $x_k$, $\chunk \beta1q$ and $\chunk \alpha1p$,
 respectively. Identifiability is considered in \cite{AGOSTO2016640}
 by specifying $\gamma$ as $\gamma=\chunk \gamma1d\in\rsetp^d$ for
 some positive integer $d$ (which corresponds to $d_x$ in
 \cite{AGOSTO2016640}) and $f(v,\gamma)$ as being of the form
 \begin{equation}
   \label{eq:Parx-f-specified}
 f(v,\gamma)=\sum_{i=1}^d\gamma_if_i(v)\;,
 \end{equation}
 for some known functions $f_1,\dots,f_d:\Vset\to\rsetp$. It is in
 fact imposed in \cite{AGOSTO2016640} that
 $v=\chunk v1d\in\rset^d=\Vset$ and $f_i(v)$ actually is a function of
 $v_i$ for each $i\in\{1,\dots,d\}$ but this constraint can be dropped for
 achieving wider generality without additional theoretical
 difficulties.  The specific form of $f(v,\gamma)$
 in~(\ref{eq:Parx-f-specified}) amounts in our setting to specify the
 previous ODMX($p,q,1)$ with reduced link
 function as in~(\ref{eq:Parx-link-general}) to a VLODMX($p,q,1,1,d+1$) with
 $\Upsilon(y,v)=(y,f_1(v),\dots,f_d(v))\in\Uset=\rset^{1+d}$,
 $A_k(\theta)= a_k$ for $k=1,\dots,p$,
 $B_1(\theta)=\begin{bmatrix} b_1&\gamma_1&\dots&\gamma_d
 \end{bmatrix}$
and $B_k(\theta)= \begin{bmatrix}
   b_k&0&\dots&0
 \end{bmatrix}
 $ for $k=2,\dots,q$.  Then
 $\theta=(\omega,a_1,\dots,a_p,b_1,\dots,b_q,\gamma)$ with
 $\gamma=\chunk\gamma1d\in\rsetp^d$ and it follows that $\Theta$ is a
 subset of $\rsetp^{1+p+q+d}$.

\section{Main results}\label{sec:main-results}

\subsection{General setting}
To investigate the identifiability of the model, we first introduce an
assumption which says how much can be identified from a single observation of the
conditional distribution $G^\theta(x,\cdot)$.
  \begin{hyp}{B}
\item \label{item:ident:G:gen:od:eq}   For all $\thv\in\Theta$ there
  exists $[\thv]_G\subset\Theta$ such that, for all $\theta\in\Theta$ and $x,x'\in\Xset$,
 \begin{equation*}
 G^\theta(x;\cdot) = G^{\thv}(x';\cdot)\quad\text{if and only if}\quad \theta\in[\thv]_G \quad\text{and}\quad x=x'\;.
 \end{equation*}
\end{hyp}
It can be convenient to write the parameters as
$\theta=(\vartheta,\varphi)$ so that $G^\theta$ only depends on
$\varphi$, hence can be denoted by $G^\varphi$, and the link function
$\psi^\theta$ only depends on $\vartheta$, hence can be denoted by
$\psi^\vartheta$.  In this case, the ``if''
in~\ref{item:ident:G:gen:od:eq} holds by setting
$[\thv]_G=\set{(\vartheta,\varphi)\in\Theta}{\varphi=\varphi_\star}$
for $\thv=(\vartheta_\star, \varphi_\star)$, and the ``only if''
in~\ref{item:ident:G:gen:od:eq} says that
$(\varphi,x)\mapsto G^\varphi(x,\cdot)$ is one-to-one. In many
examples $G^\theta$ does not depend on $\theta$ at all, in which case
$[\thv]_G=\Theta$. See \ref{item:ident:standard:G:unknown} below in
\autoref{sec:standard-lodms-appli} for such a case.

Our approach to establish identifiability is given by the following general result.
\begin{proposition}\label{thm:strong:identifiability:gen:od}
Consider an ODM$(p,q)$ satisfying
\ref{assum:gen:identif:unique:pi:gen} with $p,q\in\zsetpnz$ or an  ODMX$(p,q,r)$ satisfying
\ref{assum:gen:identif:unique:pi:gen:X} with $p,q,r\in\zsetpnz$.
 Let
  $\set{\tf[\theta]{\cdot}}{\theta\in\Theta}$ be a class of
  $\Uset^{\zsetn}\to\Xset$-measurable functions
  satisfying~\ref{item:X:from:Y:determ}. Suppose moreover
  that~\ref{item:ident:G:gen:od:eq} holds.  Then, for all
  $\thv\in\Theta$.  we have
  $$
  [\thv]=[\thv]_G\,\cap\,\langle\thv\rangle\;,
  $$
  where $\langle\thv\rangle$,  $[\thv]_G$ and  $[\thv]$ are
  respectively defined in \autoref{def:ident:theta:gen:od:link},
  Assumption~\ref{item:ident:G:gen:od:eq} and
  \autoref{def:equi:theta:gen:od}.
\end{proposition}
The proof is postponed to \autoref{sec:proof-converse-inclusion} for
convenience. We now derive the main result of this section, which provides
sufficient conditions in order to fully describe the set
$\langle\thv\rangle$. To this end, we introduce the following
assumption, in which, by saying that a probability measure $\mu$ on $(\Uset,\Usigma)$ is
\emph{non-degenerate} with respect to the class
$\mathcal{C}\subset\Usigma$, we mean that, for any $A\in\mathcal{C}$,
$\mu(A)=1$ can only be true if $A=\Uset$.
\begin{hyp}{A}
\item\label{item:ident-nu-nondegenerate} For all $\theta\in\Theta$ and
  $x\in\Xset$, the measure $\tilde G^\theta(x;\cdot)$ defined
  by~(\ref{eq:def:tildeG}) on
  $(\Uset,\Usigma)$ is non-degenerate with respect to the class
  $\mathcal{C}^\theta$,
\end{hyp}
where $\mathcal{C}^\theta$ denotes the class containing all sets $A\in\Usigma$ for which there exist
  $\theta,\theta'\in\Theta$, $z\in\Zset$, $k,l\in\zsetp$, and
  $v,w\in\Uset^k\times\Uset^l$ such that
$$
A=\set{u\in\Uset}{\tf{(v,u,w)}(z)=\tf[\theta']{(v,u,w)}(z)=0}\;.
$$
\begin{remark}\label{rem:ident-non-degenerate}
  The non-degenerate assumption~\ref{item:ident-nu-nondegenerate} is
  easy to check in the two following cases.
\begin{enumerate}[label=(\arabic*)]
\item\label{item:rem:ident-non-degenerate1} If for all
  $\theta\in\Theta$, $x\in\Xset$ and $u\in\Uset$, we have
  $\tilde G^\theta(x,\{u\})>0$, then for any set $A\in\Usigma$ we have
  $\tilde G^\theta(x,A)=1$ if and only if $A=\Uset$. Thus
  \ref{item:ident-nu-nondegenerate} is immediately satisfied.
\item\label{item:rem:ident-non-degenerate2} In the VLODM case, that
  is, with reduced link function given
  by~(\ref{eq:tilde-psi:affine:vector:case}), we immediately see that
  $\mathcal{C}^\theta$ only contains affine subsets (being the null
  space of an affine function). Hence we only need to require that
  $\tilde G^\theta(x;\cdot)$ does not have full measure on affine
  hyperplanes to ensure that it is non-degenerate with respect to the class
  $\mathcal{C}^\theta$.
\end{enumerate}
\end{remark}
In the case of an ODMX, the definition of $\tilde
G^\theta$ is different, see Remark~\ref{rem:rem-after-def}\ref{item:rem-item:tildeG},
Assumption~\ref{item:ident-nu-nondegenerate}
has to be adapted into the following.
\begin{addonehypprim}{A}
\item \label{item:ident-nu-nondegenerate-X}
For all $\theta\in\Theta$ and
  $w\in\Xset\times\Vset^r$, the measure $\tilde G^\theta(w;\cdot)$ defined
  by~(\ref{eq:def:tildeG:X}) on
  $(\Uset,\Usigma)$ is non-degenerate with respect to the class
  $\mathcal{C}^\theta$.
\end{addonehypprim}
We have the following result.
\begin{theorem}\label{thm:general-setting-main-result}
  Consider an ODM$(p,q)$ satisfying
  \ref{assum:gen:identif:unique:pi:gen} and \ref{item:ident-nu-nondegenerate} with $p,q\in\zsetpnz$ or an
  ODMX$(p,q,r)$ satisfying \ref{assum:gen:identif:unique:pi:gen:X} and \ref{item:ident-nu-nondegenerate-X}
  with $p,q,r\in\zsetpnz$.  Assume that
  \ref{item:CCMShyp}--\ref{item:continuity-in-x-for-u-given} hold. For all $\theta\in\Theta$,
  define $\mathrm{D}^\theta$ and $\tf{\cdot}$
  by~(\ref{eq:definition:f:asalimit:with:domain}).
  Then \ref{item:X:from:Y:determ} holds and we have, for all
  $\thv\in\Theta$,
  \begin{equation}
    \label{eq:thm-ident-general-odm}
  \langle\thv\rangle=\set{\theta\in\Theta}{\tilde\Psi^\theta_u(z)=\tilde\Psi^{\thv}_u(z)\text{
    for all }(z,u)\in \mathrm{E}^\thv\times\Uset}\;,
\end{equation}
where $\langle\thv\rangle$ and $\mathrm{E}^\thv$ are as in
\autoref{def:ident:theta:gen:od:link} and \autoref{def:E-sets}.
 \end{theorem}
 \begin{proof} The fact that \ref{item:X:from:Y:determ} holds for the
   given choice of $\tf{\cdot}$ follows from \autoref{lem:Lip:XY-cond}.

   Let us now take $\thv\in\Theta$ and
   $\theta\in\langle\thv\rangle$ and show that $\theta$ belongs to the
   right-hand side of~(\ref{eq:thm-ident-general-odm}).
We prove this in the case of an ODMX
satisfying~\ref{assum:gen:identif:unique:pi:gen:X}.(The case of an
ODM is readily obtained by removing the variables $V_k$'s in the
reasoning).
   By~(\ref{eq:identif-XY}),
   (\ref{eq:to-prove-deidentif:gen:od})
   and~(\ref{eq:identif-XYa-thv}), and since $\PP^\thv$ is stationary,
   we
   have, for all $t\in\zset$,
  $$
  X_{t+1} =\tilde \psi^{\thv}_{\chunk U{(t-q+1)}{t}}\left(\chunk
    X{(t-p+1)}{t}\right)= \tilde \psi^{\theta}_{\chunk
    U{(t-q+1)}{t}}\left(\chunk X{(t-p+1)}{t}\right)
  \qquad \PP^\thv\as
  $$
  Using (\ref{eq:Zk:def}) and the notation
  introduced in \autoref{sec:iter-link-funct}, we get that, for any
  $n\in\zsetp$,
  \begin{equation}
    \label{eq:iterative-theta-thv}
    \tf[\thv]{\chunk U0n}(Z_0)=    \tf[\theta]{\chunk U0n}(Z_0)
  \qquad \PP^\thv\as
\end{equation}
We now show that this implies
\begin{enumerate}[label=(H$_k$)]
\item For all $u\in\Uset^k$, we have $\displaystyle
  \tf[\thv]{(\chunk U0{(n-k)},u)}(Z_0)=
        \tf[\theta]{(\chunk U0{(n-k)},u)}(Z_0)$
  $\PP^\thv\as$
\end{enumerate}
by iterative reasoning on $k=0,\dots,n+1$. First observe
that~(\ref{eq:iterative-theta-thv}) corresponds to H$_0$ (since $u$ is
an empty sequence). Now assume that   H$_k$ holds for some
$k=0,\dots,n$. Then, for any $v\in\Uset^k$, the set
$$
A:=\set{u\in\Uset}{  \tf[\thv]{(\chunk U0{(n-k-1)},u,v)}(Z_0)=
  \tf[\theta]{(\chunk U0{(n-k)},u,v)}(Z_0)}
$$
has probability 1 under the $\PP^\thv$-conditional probability of
$U_{n-k}$ given $Z_0,\chunk U0{(n-k-1)},\chunk
V{(-\infty)}{(n-k-1)}$. This conditional probability is
$\tilde G^{\thv}((X_{n-k},\chunk
V{(n-k-r)}{(n-k-1)});\cdot)$ defined
by~(\ref{eq:def:tildeG:X}). By~\ref{item:ident-nu-nondegenerate-X} and
since, given $Z_0,\chunk U0{(n-k-1)},\chunk
V{(-\infty)}{(n-k-1)}$, we have $A\in\mathcal{C}^\theta$,
$\PP^\thv\as$, we obtain that H$_{k+1}$ is true. Reasoning by
induction, this leads to H$_{n+1}$, and finally, we
get that
\begin{enumerate}[label=(H)]
\item there exists $z\in\Zset$ such that for all $n\in\zsetp$ and all
$u\in\Uset^{n+1}$,
$\tf[\thv]{u}(z)=\tf[\theta]{u}(z)$.
\end{enumerate}
Now let $(\initmle,u_1)\in \mathrm{E}^\thv\times\Uset$. By definition of
$\mathrm{E}^\thv$, there exists $u\in\Uset^{\zsetn}$, such that $\initmle=\tF[\thv]{u}$.
Now, for all $n\geq p$, we have
\begin{align*}
\tilde\Psi^\theta_{u_1}\left(\tF{\chunk u{(-n)}0}(z)\right)&=
                                              \tF{\chunk u{(-n)}1}(z)\\
  &=\tF[\thv]{\chunk u{(-n)}1}(z)\\
&=\tilde\Psi^\thv_{u_1}\left(\tF[\thv]{\chunk u{(-n)}0}(z)\right)\;,
\end{align*}
where we chose $z\in\Zset$ in order to apply Assertion~(H) in the
second equality.  On the other hand, by~(\ref{eq:definition:fF:asalimit:with:domain}), we have
$$
\initmle=\tF[\thv]{u}=\lim_{n\to\infty}\tF[\thv]{\chunk u{(-n)}0}(z)=\lim_{n\to\infty}\tF[\theta]{\chunk u{(-n)}0}(z)\;,
$$
where we again used that $z$ were chosen as in (H) in the last
equality.  With \ref{item:continuity-in-x-for-u-given} and the
previous display we obtain
$\tilde\Psi^\theta_{u_1}(\initmle)=\tilde\Psi^\thv_{u_1}(\initmle)$. This
is true for an arbitrary $(\initmle,u_1)\in \mathrm{E}^\thv\times\Uset$;
hence, we have obtained that the left-hand side
of~(\ref{eq:thm-ident-general-odm}) is included in its right-hand
side.

We now prove the opposite inclusion. Let $\thv,\theta\in\Theta$ such that
\begin{equation}
  \label{eq:theta-well-chosen-E}
\tilde\Psi^\theta_u(z)=\tilde\Psi^{\thv}_u(z)\qquad\text{for all}\quad
(z,u)\in\mathrm{E}^\thv\times\Uset\;.
\end{equation}
Let $v\in\Uset^{\zsetn}$.  Take an arbitrary
$\initmle\in\mathrm{E}^\thv$. Then there exists
$w\in\mathrm{D}^\theta$ such that $\initmle=\tF{w}$. For 
all  $n\in\zsetpnz$ and $k=1,\dots,n$, we get that
$$
\tF[\thv]{\chunk
  v{(-n+1)}{(-n+k)}}(\initmle)=\tF[\thv]{(w,\chunk
  v{(-n+1)}{(-n+k)})}\in\mathrm{E}^\thv\;.
$$
Applying~(\ref{eq:theta-well-chosen-E})
recursively in $k$, we get that, for any $n\in\zsetp$,
  $$
  \tF[\thv]{\chunk v{(-n)}0}(\initmle)=\tF[\theta]{\chunk v{(-n)}0}(\initmle)\;.
  $$
  Hence, $\mathrm{D}^\theta=\mathrm{D}^\thv$ and by~(\ref{eq:definition:fF:asalimit:with:domain})
  and~(\ref{eq:psi:Psi:relation;infty}), we get that for all
  $v\in\mathrm{D}^\theta=\mathrm{D}^\thv$,
  $\tf[\theta]{v}=\tf[\thv]{v}$. By \autoref{lem:Lip:XY-cond} we have
  $\chunk U{(-\infty)}0\in\mathrm{D}^\thv$, $\PP^\thv\as$, and
  using~(\ref{eq:to-prove-deidentif:def-angle-set-limit-case}), we get
  that $\theta\in\langle\thv\rangle$, which concludes the proof.
\end{proof}
Note that in~(\ref{eq:thm-ident-general-odm}), as in the left-hand
side of~(\ref{eq:ident_fundamental_E}), the functions
$\tilde\Psi^\theta_u$ and $\tilde\Psi^{\thv}_u$ are only required to
coincide on $\mathrm{E}^\thv$ whereas in the left-hand side
of~(\ref{eq:ident_fundamental}), they coincide
on the whole set $\Zset$. In some cases, we can prove that the two
conditions are the same, so that
\autoref{thm:strong:identifiability:gen:od} and
\autoref{thm:general-setting-main-result} allow us to conclude that
the inclusion in~(\ref{eq:ident_fundamental}) is in fact an equality,
as in the following result.
 \begin{corollary}\label{thm:the-main-result-general-setting}
   Suppose that the assumptions of
   \autoref{thm:general-setting-main-result} and
   \ref{item:ident:G:gen:od:eq} hold. Let $\thv\in\Theta$. Then the
   inclusion in~(\ref{eq:ident_fundamental_E}) is an equality. Suppose
   moreover that $\thv$
   satisfies the following additional assumption. 
   \begin{hyp}{A}
   \item \label{assum:fundamental-not-always-true-assumption}
For all $\theta\in\Theta$, $u\in\Uset$ and $z\in\Zset$,
   if $\tilde \Psi^\theta_u$ and $\tilde \Psi^\thv_u$ coincide on the set
   $\mathrm{E}^\thv(z):=\set{\tF[\thv]{v}(z)}{n\in\zsetp,\,v\in\Uset^n}$, then they also
   coincide on $\Zset$.
 \end{hyp}
 Then the inclusion
   in~(\ref{eq:ident_fundamental}) is an equality.
\end{corollary}
\begin{proof}
  Applying \autoref{thm:strong:identifiability:gen:od}
  and  \autoref{thm:general-setting-main-result}, we get that
  $$
  [\thv]=\set{\theta\in[\thv]_G}{{\tilde\Psi^\theta_u(z)=\tilde\Psi^{\thv}_u(z)\text{
    for all }(z,u)\in \mathrm{E}^\thv\times\Uset}}\;.
  $$
  Observing that, by~\ref{item:ident:G:gen:od:eq}, $\theta\in[\thv]_G$
  is equivalent to have $G^\theta=G^\thv$, we get that the
   inclusion in~(\ref{eq:ident_fundamental_E}) is an equality. 
  To prove the second assertion of the corollary, we only need to check
  that, under~\ref{assum:fundamental-not-always-true-assumption}, for all $\theta\in\Theta$ and $u\in \Uset$, if
  $\tilde\Psi^\theta_u$ and $\tilde\Psi^{\thv}_u$ coincide on
  $\mathrm{E}^\thv$, they must also coincide on
  $\Zset$. It suffices to
  show that there exists $z\in\Zset$ such
  that $\mathrm{E}^\thv(z)\subset\mathrm{E}^\thv$.
  This inclusion is true if
  $z\in\mathrm{E}^\thv$ and we conclude by observing that
  $\mathrm{E}^\thv$ is not empty since it contains $Z_1$,
  $\PP^\thv\as$, as a consequence of \autoref{rem:E-theta-Zk}. 
\end{proof}
\begin{remark}
  A simple case
  where~\ref{assum:fundamental-not-always-true-assumption} in
  \autoref{thm:the-main-result-general-setting} is easy to check is
  when $p=q=1$, so that $\Zset=\Xset$ and
  $\tF[\thv]{v}(z)=\tilde{\psi}^\thv_{v}(z)$ for all $v\in\Uset$ and
  $z\in\Xset$. See the proof of \autoref{thm:setpar} for a specific
  example.  However it may happen that
  \ref{assum:fundamental-not-always-true-assumption} is not satisfied
  as will be seen in
  \autoref{rem:standard-lodms}\ref{item:exple-strict-inclusion}. In the linear case, we will characterize
  $\langle\thv\rangle$ in \autoref{lem:vlodm:angle-set} without
  relying on  \ref{assum:fundamental-not-always-true-assumption}.
\end{remark}

\subsection{Vector linear setting}
\label{sec:vect-line-setting}
We now consider a VLODM($p,q,p',q'$) or a VLODMX($p,q,p',q',r$), that
is, we assume the reduced link function to be of the form~(\ref{eq:tilde-psi:affine:vector:case}). We set in this case
$\Xmet(x,x')=|x-x'|$ (resp. $\Umet(u,u')=|u-u'|$) where $|\cdot|$
denotes an arbitrary norm in $\rset^{p'}$ (resp. $\rset^{q'}$). The
general conditions reduce to the following set of conditions.
\begin{hyp}{L}
\item\label{item:vlodm-ident-cond-cond-dens-def} For all
  $\theta\in\Theta$, we have that $\mathrm{I}_{p'}
  -\sum_{k=1}^pA_k(\theta)z^k$
  is invertible for all
  $z\in\cset$ with $|z|\leq1$,
\end{hyp}
where $\mathrm{I}_{p'}$ denotes the identity matrix of order $p'$. 
\begin{hyp}{L}
\item\label{item:vlodm--non-degenerateident-G} For all $\theta\in\Theta$ and
  $x\in\Xset$, the measure $\tilde{G}^\theta(x;\cdot)$ defined on
  $(\Uset,\Usigma)$ by~(\ref{eq:def:tildeG}) is non-degenerate in
the following sense : there is no affine hyperplane
$A\subset\rset^{q'}$ such that $\tilde{G}^\theta(x;A)=1$.
\end{hyp}
Note that, if $q'=1$, affine hyperplanes are singletons, hence
\ref{item:vlodm--non-degenerateident-G} simply means that, for all $x$
and $\theta$,
$\tilde G^\theta(x;\cdot)$ does not reduce to a unit mass concentrated
on a single point.  In the case of a VLODMX,
we  replace \ref{item:vlodm--non-degenerateident-G}  by
the following.
\begin{addonehypprim}{L}
  \item\label{item:vlodm--non-degenerateident-G-X} For all $\theta\in\Theta$ and
  $w\in\Xset\times\Vset^r$, the measure $\tilde{G}^\theta(w;\cdot)$ defined on
  $(\Uset,\Usigma)$ by~(\ref{eq:def:tildeG:X}) is non-degenerate in
the following sense : there is no affine hyperplane
$A\subset\rset^{q'}$ such that $\tilde{G}^\theta(w;A)=1$.
\end{addonehypprim}
Finally the moment condition~\ref{item:ident-cond-log-moment}
simplifies in the vector linear case to
\begin{hyp}{L}
\item\label{item:vlodm-ident-cond-log-moment}
The  invariant probability measure of \autoref{def:ergo:theta:gen:od}
satisfies, for all $\theta\in\Theta$,
  \begin{equation}
    \label{eq:identif:abitofmoment}
\mathbb{E}^\theta\left[\ln^+ (|U_0|)\right]<\infty\;.
  \end{equation}
\end{hyp}
We have the following result, whose proof is postponed to
\autoref{sec:proof-lem:ergo:ass:1} for convenience, that relates this
set of assumptions to the general ones.
\begin{lemma}\label{lem:ergo:ass:1}
  Consider the vector linear setting
  where~(\ref{eq:tilde-psi:affine:vector:case}) holds, and $\Xset$ and
  $\Yset$ are closed subset of $\rset^{p'}$ and $\rset^{q'}$,
  respectively, with $\Xmet$ and $\Umet$ being the metrics induced by
  norms on these spaces. The following assertions hold.
  \begin{enumerate}[label=(\roman*)]
  \item\label{item:item:CCMShyp} Assumption~\ref{item:CCMShyp} holds.
  \item\label{item:assum:bound:rho:gen} Assumption~\ref{assum:bound:rho:gen} is equivalent
  to \ref{item:vlodm-ident-cond-cond-dens-def}.
\item\label{item:item:ident-cond-log-moment} Assumption
  \ref{item:vlodm-ident-cond-log-moment} implies
  \ref{item:ident-cond-log-moment} for any $\initmlex_1\in\Xset$, and any
  $\initmleu_1 \in\Uset$.
\item\label{item:item:continuity-in-x-for-u-given} Assumption~\ref{item:continuity-in-x-for-u-given} holds.
\item\label{item:item:ident-nu-nondegenerate}
  Assumption~\ref{item:vlodm--non-degenerateident-G} implies \ref{item:ident-nu-nondegenerate}.
\item\label{item:item:ident-nu-nondegenerate-X}
  Assumption~\ref{item:vlodm--non-degenerateident-G-X} implies \ref{item:ident-nu-nondegenerate-X}.
\end{enumerate}
As a consequence, the assumptions of
\autoref{thm:general-setting-main-result} are implied
by~\ref{assum:gen:identif:unique:pi:gen},~\ref{item:vlodm-ident-cond-cond-dens-def},~\ref{item:vlodm--non-degenerateident-G}
and~\ref{item:vlodm-ident-cond-log-moment} in the VLODM case and by ~\ref{assum:gen:identif:unique:pi:gen:X},~\ref{item:vlodm-ident-cond-cond-dens-def},~\ref{item:vlodm--non-degenerateident-G-X}
and~\ref{item:vlodm-ident-cond-log-moment} in the VLODMX case.
\end{lemma}
We now provides a simple characterization of $\langle\thv\rangle$
in~(\ref{eq:thm-ident-general-odm}) in the linear case. The proof of
the following result can be found
in \autoref{sec:proof-lem:vlodm:angle-set}.
\begin{lemma}\label{lem:vlodm:angle-set}
  Suppose that $\{0\}\subsetneq\Uset$ and that~\ref{item:vlodm-ident-cond-cond-dens-def} holds, and
  let $\mathrm{E}^\theta$ be as in
  \autoref{def:E-sets}.
For all $\theta\in\Theta$, define $\mathbf{R}(\cdot;\theta)$ as the rational matrix 
\begin{equation}
  \label{eq:RpDef}
\mathbf{R}(z;\theta)=\left(\mathrm{I}_{p'} z^p
  -\sum_{k=1}^pA_{k}(\theta)z^{p-k}\right)^{-1}
\left(\sum_{k=0}^{q-1} B_{k+1}(\theta)\ z^{q-1-k}\right)
\;,
\end{equation}
which is well defined on $z\in\cset$ except for at most finitely many
$z$'s.
  Then, for all $\theta,\thv\in\Theta$, the
  two following  assertions are equivalent.
  \begin{enumerate}[label=(\roman*)]
  \item   \label{item:theta-well-chosen-E-again} We have 
$\displaystyle\tilde\Psi^\theta_u(z)=\tilde\Psi^{\thv}_u(z)$ for all
$z\in\mathrm{E}^\thv$ and $u\in\Uset$. 
\item    \label{item:theta-well-chosen-E-CNS}  The two following identities hold
  \begin{align}
  \label{eq:vlodm:iden:omega-eq}
\left(\mathrm{I}_{p'}-\sum_{k=1}^pA_k(\thv)\right)^{-1}\boldsymbol{\omega}(\thv)
  &=
    \left(\mathrm{I}_{p'}-\sum_{k=1}^pA_k(\theta)\right)^{-1}\boldsymbol{\omega}(\theta)   \;,  \\
  \label{eq:vlodm:iden:R-eq}
 \mathbf{R}(\cdot;\thv) &=
                             \mathbf{R}(\cdot;\theta)  \;.
  \end{align}
  \end{enumerate}
\end{lemma}
The identification of a parameter $\theta$ based on the
equation~(\ref{eq:vlodm:iden:R-eq}) is similar to the
identifiability of a vector auto-regressive moving average or order $p,q$
(VARMA($p,q$)) model with AR matrices $\chunk
{A}1p(\theta)$ and MA matrices  $\chunk
B1q(\theta)$. Indeed, in such a model the spectral density matrix takes the form
$$
\lambda\mapsto
\mathbf{R}(\rme^{\rmi\lambda};\theta)\Sigma \mathbf{R}(\rme^{-\rmi\lambda};\theta)\;,
$$
where $\Sigma$ is the covariance matrix of the noise. We refer to
\cite{hannan_deistler2012} where identifiable parametrization of ARMA
models are discussed. Below we provide an important related result
related to this general issue.  Let $p,q,p',q'$ be positive integers.
For all $\chunk A1p\in(\rset^{p'\times p'})^p$ and $\chunk
B1q\in(\rset^{p'\times q'})^q$, let us define the polynomial matrices
respectively valued in $\rset^{p'\times p'}$ and $\rset^{p'\times
  q'}$
\begin{align}
  \label{eq:P-and-Q}
  \mathrm{P}_p(z;\chunk{A}1p)=  \mathrm{I}_{p'}z^p-\sum_{k=1}^pA_{k}z^{p-k} \quad\text{and}\quad
  \mathrm{Q}_q(z;\chunk{B}1q)=
\sum_{k=1}^{q} B_{k}\ z^{q-k}\;.
\end{align}
Note that, for all $\chunk A1p\in(\rset^{p'\times p'})^p$,
$\mathrm{P}_p(z;\chunk{A}1p)$ in~(\ref{eq:P-and-Q}) must be invertible
for $|z|$ large enough (since then $\mathrm{I}_{p'}z^p$ dominates).
When a polynomial matrix is invertible for at least one $z\in\cset$,
then it is invertible for all  $z\in\cset$, except at most a finite
number of them. It is then said to be \emph{non-singular}. Thus, for all $\chunk A1p\in(\rset^{p'\times p'})^p$ and $\chunk
B1q\in(\rset^{p'\times q'})^q$, we can define the rational matrix
$\mathrm{P}_p(z;\chunk{A}1p)^{-1} \mathrm{Q}_q(z;\chunk{B}1q)$,
which is well defined for  all  $z\in\cset$, except at most a finite
number of them.

\begin{lemma}\label{lem:canonical-vlodm}
 Let $p,q,p',q'$ be positive integers.
Then, for any   $\chunk{A^{\star}}1p\in(\rset^{p'\times p'})^p$ and $\chunk
 {B^{\star}}1q\in(\rset^{p'\times q'})^q$, the two following assertion holds.
    \begin{enumerate}[label=(\roman*)]
  \item\label{item:thm:gener-ident-cond-vlodm:cs} Suppose that
 $\mathrm{P}_p(\cdot;\chunk{A^{\star}}1p)$ and   $\mathrm{Q}_q(\cdot;\chunk{B^{\star}}1q)$
 are left coprime. Then, for all $\chunk A1p\in(\rset^{p'\times p'})^p$ and $\chunk
 B1q\in(\rset^{p'\times q'})^q$, we have
 $$
 \mathrm{P}_p(\cdot;\chunk{A^{\star}}1p)^{-1}\mathrm{Q}_q(\cdot;\chunk{B^{\star}}1q)=
 \mathrm{P}_p(\cdot;\chunk{A}1p)^{-1}\mathrm{Q}_q(\cdot;\chunk{B}1q)
 $$
 if and only if $\chunk A1p=\chunk{A^{\star}}1p$ and  $\chunk
 B1p=\chunk{B^{\star}}1p$.
\item \label{item:thm:gener-ident-cond-vlodm:cn} Suppose that
 $\mathrm{P}_p(\cdot;\chunk{A^{\star}}1p)$ and   $\mathrm{Q}_q(\cdot;\chunk{B^{\star}}1q)$
 are not left coprime. Then, there exist
 $\chunk {\tilde A}1p\in(\rset^{p'\times p'})^p\setminus\{0\}$ and
 $\chunk {\tilde B}1q\in(\rset^{p'\times q'})^q$
 such that, for all $\alpha\in\rset$, setting $\chunk{A}1p=\chunk{A^{\star}}1p+\alpha\chunk {\tilde
   A}1p$ and $\chunk
 {B}1q=\chunk
 {B^{\star}}1q+\alpha\chunk {\tilde B}1q$, we have
 $$
 \mathrm{P}_p(\cdot;\chunk{A^{\star}}1p)^{-1}\mathrm{Q}_q(\cdot;\chunk{B^{\star}}1q)=
 \mathrm{P}_p(\cdot;\chunk{A^{\star}}1p+\alpha\chunk {\tilde
   A}1p)^{-1}\mathrm{Q}_q(\cdot;\chunk
 {B^{\star}}1q+\alpha\chunk {\tilde B}1q)\;.
 $$
\end{enumerate}
\end{lemma}
Two polynomial matrices with the same number $p'$ of rows are said to be left coprime if they admit
the identity matrix $\mathrm{I}_{p'}$ as a greatest common $p'\times p'$ left divisor
(g.c.l.d.), that is, every common left divisor of them is also a left
divisor of $\mathrm{I}_{p'}$.  The set of polynomial matrices of order
$p'$ is a non-commutative ring for $p'>1$. This is why for $p'>1$ a
notion of \emph{left} (or right) divisor is necessary. Note, however
that if $p'=1$, saying that $\mathrm{P}_p(\cdot;\chunk{A^{\star}}1p)$ and
$\mathrm{Q}_q(\cdot;\chunk{B^{\star}}1q)$ are \emph{left} coprime is
equivalent to say that $\mathrm{P}_p(\cdot;\chunk{A^{\star}}1p)$ and the
$q'$ row entries of $\mathrm{Q}_q(\cdot;\chunk{B^{\star}}1q)$ (which is
$q'$-dimensional row vector of polynomials of degree at most $q$) are
coprime, that is, they have 1 as greater common divisor. In particular
if $p'=q'=1$, this boils down to say that
$\mathrm{P}_p(\cdot;\chunk{A^{\star}}1p)$ and
$\mathrm{Q}_q(\cdot;\chunk{B^{\star}}1q)$ have no common roots.  The case
$p'>1$ is significantly more complicated and we refer to
\cite[Chapter~III]{macduffee33} for an excellent introduction on
polynomials on Euclidean rings that applies to matrices of
polynomials.
\begin{proof}
 Let    $\chunk{A^{\star}}1p\in(\rset^{p'\times p'})^p$ and $\chunk
 {B^{\star}}1q\in(\rset^{p'\times q'})^q$. In this proof section, for
 convenience we denote $\mathrm{P}_p(\cdot;\chunk{A^{\star}}1p)$ and
 $\mathrm{Q}_q(\cdot;\chunk{B^{\star}}1q)$ by $P^{\star}$ and $Q^{\star}$.

\noindent\textbf{Proof of Assertion~\ref{item:thm:gener-ident-cond-vlodm:cs}.}
Suppose that $P^{\star}$ and $Q^{\star}$ are left coprime.
The Bezout theorem for matrices of polynomials (see
e.g. \cite[Theorem~3.1]{macduffee33}) gives that there exists two
polynomial matrices $R$ and $S$ of order $p'\times p'$ and  $q'\times
p'$ respectively, such that
\begin{equation}
  \label{eq:bezout}
\mathrm{I}_{p'}= P^{\star}R+Q^{\star}S\;.
\end{equation}
Let
$\chunk A1p\in(\rset^{p'\times p'})^p$ and
$\chunk B1q\in(\rset^{p'\times q'})^q$, and denote $P=\mathrm{P}_p(\cdot;\chunk{A^{\star}}1p)$ and
$Q=\mathrm{Q}_q(\cdot;\chunk{B^{\star}}1q)$. The ``if'' in
Assertion~\ref{item:thm:gener-ident-cond-vlodm:cs} is obvious, so we
only need to assume that
\begin{equation}
  \label{eq:cn-only-if}
 P^{-1}Q=P^{*-1}Q^{\star}
\end{equation}
and prove that $P=P^{\star}$
(in which case we also get that $Q=Q^{\star}$).
Define the rational matrix $U=PP^{*-1}$. Multiplying both sides
of~(\ref{eq:bezout}) by $U$ from the left, we have
$$
U= PR+PP^{*-1}Q^{\star}S=PR+P P^{-1}QS=PR+QS\;,
$$
where we used~(\ref{eq:cn-only-if}) in the second equality. Hence we
get that $U$ is a polynomial matrix and since $UP^{*}=P$ and both
$P^{*}$ and $P$ are of the form $\mathrm{I}_{p'}z^p+$ a polynomial of
degree at most $p-1$, we get that $U=\mathrm{I}_{p'}$ and so
$P^{*}=P$, which concludes the proof
of~\ref{item:thm:gener-ident-cond-vlodm:cs}.

\noindent\textbf{Proof of
  Assertion~\ref{item:thm:gener-ident-cond-vlodm:cn}.}
Suppose that $P^{\star}$ and $Q^{\star}$ are not left coprime.  Let $D$ be a
g.c.l.d. of $(P^{\star},Q^{\star})$.  Then $D$ is a polynomial matrix that
left-divides $P^{\star}$ and $Q^{\star}$ and changing this polynomial matrix won't
change the rational matrix $P^{*-1}Q^{\star}$. The difficulty is to show
that we can modify a left divisor of $P^{\star}$ and $Q^{\star}$ in such a way
that the resulting $P$ and $Q$ are still of the form~(\ref{eq:P-and-Q})
for some well chosen $\chunk A1p\in(\rset^{p'\times p'})^p$ and
$\chunk B1q\in(\rset^{p'\times q'})^q$. To this end we must first
choose $D$ in a special form. Indeed, for any \emph{unimodular}
polynomial matrix $U$, $DU$ is also a g.c.l.d. of $(P^{\star},Q^{\star})$. The
polynomial matrix $DU$ is called a right-associate of $D$, and by
\cite[Theorem~22.1]{macduffee33}, we can choose $U$ so that $DU$ is in
\emph{Hermite normal form}, that is such that $DU$ is triangular
inferior,
$$
DU=  \begin{bmatrix}
  h_{1,1}&0&\dots&0\\
  h_{2,1}&h_{2,2}&\ddots&\vdots\\
  \vdots&\ddots&\ddots&0\\
  h_{p',1}&\dots&\dots&h_{p',p'}
\end{bmatrix}
$$
with the polynomial $h_{i,i}$ is unitary with degree
$\mathrm{deg}(h_{i,i})$ strictly larger
than those of $h_{j,i}$ for all $j>i$ (that is, the degree on the diagonal
dominates those of the same column). Let
$$
k=\min\set{i=1,\dots,p'}{\mathrm{deg}(h_{i,i})>0}\;.
$$
From what precedes, this min exists (the set is not empty), otherwise
we would have $DU=\mathrm{I}_{p'}$ and $P^{\star}$ and $Q^{\star}$ would be left
coprime. We can thus write, for some $k\in\{1,\dots,p'\}$,
$$
DU=
\begin{bmatrix}
  \mathrm{I}_{k-1} &   \mathrm{0}_{k-1,p'-k+1} \\
  \mathrm{0}_{p'-k+1,k-1}
  & T
\end{bmatrix}\quad\text{with}\quad
T= \begin{bmatrix}
  h_{k,k}&0&\dots&0\\
    \vdots&\ddots&\ddots&\vdots\\
  \vdots&&\ddots&0\\
  h_{p',k}&\dots&\dots&h_{p',p'}
\end{bmatrix}
\;,
$$
where $\mathrm{0}_{k,\ell}$ is the zero matrix of size
$k\times\ell$. By convention, if $k=1$, $DU$ reduces to
$T$, that is back to its previous form. The important
point is that with this definition of $k$, we know that $h_{k,k}$ is
unitary with $\mathrm{deg}(h_{k,k})\geq1$. Now, since  $DU$ is a left divisor of
$P^{\star}$ and $Q^{\star}$ we may write
$$
P^{\star}=(DU)R\quad\text{and}\quad Q^{\star}=(DU)S
$$
for some matrices $R$ and $S$ of respective sizes $p'\times p'$ and
$p'\times q'$. We write $R$ and $S$ in a block matrix form compatible
with that of $DU$, that is
$$
R=\begin{bmatrix}
  R^{(1)}\\
    R^{(2)}
  \end{bmatrix}
  \quad\text{and}\quad
S=\begin{bmatrix}
  S^{(1)}\\
    S^{(2)}
  \end{bmatrix}
  $$
  with $R^{(1)}$ and $S^{(1,1)}$ of respective sizes
  $(k-1)\times(k-1)$ and $(k-1)\times q'$. (again with convention that
  these matrices vanish if $k=1$). Then we have
  $$
  P^{*}=\begin{bmatrix}
  R^{(1)}\\
    TR^{(2)}
  \end{bmatrix}
  \quad\text{and}\quad
Q=\begin{bmatrix}
  S^{(1)}\\
   T S^{(2)}
  \end{bmatrix}
  $$
  In particular the first row of $TR^{(2)}$ is $h_{k,k}$ times the
  first row of $R^{(2)}$, and since $h_{k,k}$ is unitary with
  $\mathrm{deg}(h_{k,k})\geq1$, the form of $P^{\star}$ implies that the
  first row of $R^{(2)}$ is made of polynomials of degrees at most
  $p-1$ and cannot be zero (since the degree of the $(k,k)$ entry row
  of $P^{\star}$ is exactly $p$ and all the other entries are
  zero). Similarly, the first row of $S^{(2)}$ is made of polynomials
  of degrees at most $q-2$ (since $Q^{\star}$ is of degree $q-1$).  Let
  $\Delta_k$ denote the diagonal matrix od order $p'$ with zeros on
  its diagonal except on the $k$-th entry where it is 1. Since $\Delta_kR$ and $\Delta_kS$
  only keeps the first rows of the block matrices $R^{(2)}$ and
  $S^{(2)}$, respectively, and put all other entries to zero, we get
  from what precedes that $\Delta_kR$ and $\Delta_kS$ are of degree at
  most $p-1$ and $q-1$, respectively, and that $\Delta_kR$ is not
  zero. Hence we may find
  $\chunk {\tilde A}1p\in(\rset^{p'\times p'})^p\setminus\{0\}$ and
  $\chunk {\tilde B}1q\in(\rset^{p'\times q'})^q$ (with
  $\tilde B_1=0$) such that
  $$
  \Delta_kR=\mathrm{Q}_p(z;\chunk{\tilde A}1p)
  \quad\text{and}\quad \Delta_kS=\mathrm{Q}_q(z;\chunk{\tilde B}1q)\;.
  $$
  Then, for all  $\alpha\in\rset$, setting
  $\chunk{A}1p=\chunk{A^{\star}}1p+\alpha\chunk {\tilde A}1p$ and
  $\chunk {B}1q=\chunk {B^{\star}}1q+\alpha\chunk {\tilde B}1q$, we have
$$
\mathrm{P}_p(z;\chunk{A}1p)=\mathrm{P}_p(z;\chunk{A^{\star}}1p)+\alpha
\Delta_kR = (DU+\alpha \Delta_k)R\;,
$$
and, similarly,
$$
\mathrm{Q}_p(z;\chunk{B}1p)=\mathrm{Q}_q(z;\chunk{B^{\star}}1q)+\alpha
\Delta_kS = (DU+\alpha \Delta_k) S\;.
$$
Then we get
\begin{align*}
 \mathrm{P}_p(\cdot;\chunk{A}1p)^{-1}\mathrm{Q}_q(\cdot;\chunk{B}1q)&=
\left((DU+\alpha \Delta_k)R\right)^{-1}(DU+\alpha \Delta_k)
                                                                      S=R^{-1}S\\
                                                                    &=\left(DUR\right)^{-1}DUS\\
  &=P^{*-1}Q^{\star}
\;,
\end{align*}
which concludes the proof of
Assertion~\ref{item:thm:gener-ident-cond-vlodm:cn}.
\end{proof}
\section{Applications}
\label{sec:application}
We now use our results to derive necessary and sufficient conditions for having
identifiability in  the examples of \autoref{sec:examples}. Many other
examples can be achieved by combining various observation kernels and
link function with or without exogenous covariates.

\subsection{Standard LODMs}
\label{sec:standard-lodms-appli}
Let us apply the results of \autoref{sec:vect-line-setting} in
the case of standard LODMs as defined in \autoref{def:standard-lodms}.
The moment assumption~\ref{item:vlodm-ident-cond-log-moment} can be readily used
for  standard LODMs with known or unknown observation kernel, with
$|\cdot|$ in~(\ref{eq:identif:abitofmoment}) denoting the usual
absolute value.
The other assumptions of the general VLODM listed in
\autoref{sec:vect-line-setting} can be simplified as follows.

For a standard LODM,
Assumption~\ref{item:vlodm-ident-cond-cond-dens-def} becomes
\begin{hyp}{SL}
\item\label{item:lodm-lipschitz-standard} For all
  $\theta=(\omega,\chunk{a}{1}{p},\chunk{b}{1}{q})$, we have
  $1-\sum_{k=1}^pa_kz^k\neq0$ for all $z\in\cset$ such that $|z|\leq1$.
\end{hyp}
In the case of  a standard LODM with unknown observation kernel it
becomes
\begin{addonehypprim}{SL}
\item\label{item:lodm-lipschitz-standard-unknownG} For all
  $\theta=(\vartheta,\varphi)\in\Theta$ with
  $\vartheta=(\omega,\chunk{a}{1}{p},\chunk{b}{1}{q})$, we have
  $z-\sum_{k=1}^pa_kz^k\neq0$ for all $z\in\cset$ such that
  $|z|\leq1$.
\end{addonehypprim}
As
for~\ref{item:vlodm--non-degenerateident-G}, it becomes
\begin{hyp}{SL}
\item\label{item:lodm-degenerate-standard} For all $x\in\Xset$,
  $\Upsilon(Y)$ does not degenerate to a single point for
  $Y\sim G(x,\cdot)$, that is, for all
  $u\in\rset$, we have $G(x,\{\Upsilon(\cdot)=u\})<1$;
\end{hyp}
and, in case of an unknown observation kernel,
\begin{addonehypprim}{SL}
\item \label{item:lodm-degenerate-standard-unknown-case} For all  $\varphi\in\Phi$,
  $x\in\Xset$ and
  $u\in\rset$, we have $G^{\varphi}(x,\{\Upsilon(\cdot)=u\})<1$.
\end{addonehypprim}
Finally \ref{item:ident:G:gen:od:eq} becomes
\begin{hyp}{SL}
\item \label{item:ident:standard:G}   For all $x,x'\in\Xset$,
 $G(x;\cdot) = G(x';\cdot)$ if and only if $x=x'$;
\end{hyp}
and, in the case of an unknown observation kernel, it
reads as
\begin{addonehypprim}{SL}
\item \label{item:ident:standard:G:unknown}   For all
  $\theta=(\vartheta,\varphi)$ and
  $\thv=(\vartheta_{\star},\varphi_{\star})$ in $\Theta$, for all
  $x,x'\in\Xset$, we have
 \begin{equation*}
 G^\varphi(x;\cdot) = G^{\varphi_\star}(x';\cdot)\quad\text{if and only if}\quad \varphi=\varphi_\star \quad\text{and}\quad x=x'\;.
 \end{equation*}
\end{addonehypprim}
This says that the class $[\thv]_G$
in~\ref{item:ident:G:gen:od:eq} is given by
$[\thv]_G=\set{\theta=(\vartheta,\varphi)\in\Theta}{\varphi=\varphi_{\star}}$.

Remarkably, all LODMs share the same necessary and
sufficient condition for identifiability, which can be expressed as follows, using
$\chunk{a^\star}1p$ and $\chunk{b^\star}1q$ to denote the true linear
coefficients of the linear link function.
\begin{hyp}{SL}
\item\label{item:lodm-ident-poly} The polynomials $z^p-\sum_{k=1}^pa^{\star}_{k}z^{p-k}$ and
    $\sum_{k=0}^{q-1} b^{\star}_{k+1}\ z^{q-1-k}$
have no common complex roots.
\end{hyp}
We can now state the following result, which says that the true
parameter $\thv$ in the interior of $\Theta$ is identifiable if and only
if~\ref{item:lodm-ident-poly} holds.
\begin{theorem}
  \label{thm:gener-ident-cond}
  Consider a standard LODM$(p,q)$
  satisfying~\ref{assum:gen:identif:unique:pi:gen}
  and~\ref{item:vlodm-ident-cond-log-moment} for some
  $p,q\in\zsetpnz$, and suppose that $0\in\Uset$. In the case of a
  known observation kernel, suppose that
  \ref{item:lodm-lipschitz-standard}--\ref{item:ident:standard:G}
  hold.   In the case of an
  unknown observation kernel, suppose that
  \ref{item:lodm-lipschitz-standard-unknownG}--\ref{item:ident:standard:G:unknown} hold.
  Then the inclusion in~(\ref{eq:ident_fundamental_E}) is an
  equality for all $\thv\in\Theta$. In the case of a
  known observation kernel,
  Assertions~\ref{item:thm:gener-ident-cond:cs}
  and~\ref{item:thm:gener-ident-cond:cn} below hold for any
  $\thv=(\omega^\star,\chunk{a^\star}1p,\chunk{b^\star}1q)\in\Theta$.
In the case of a
  known observation kernel,
  Assertions~\ref{item:thm:gener-ident-cond:cs}
  and~\ref{item:thm:gener-ident-cond:cn:unkown} below hold for any
  $\thv=(\omega^\star,\chunk{a^\star}1p,\chunk{b^\star}1q,\varphi_\star)\in\Theta$.
  \begin{enumerate}[label=(\roman*)]
  \item\label{item:thm:gener-ident-cond:cs} Condition
    \ref{item:lodm-ident-poly} implies that 
    $[\thv]$ reduces to the singleton $\{\thv\}$.
  \item\label{item:thm:gener-ident-cond:cn} If Condition~\ref{item:lodm-ident-poly} does
    not hold, then there exists an open segment $I^\star\subset\rset^{1+p+q}$ of positive
    length and containing $\thv$ such that
    $I^\star\cap\Theta\subset[\thv]$.
  \item\label{item:thm:gener-ident-cond:cn:unkown} If~\ref{item:lodm-ident-poly} does
    not hold, then there exists an open segment $I^\star\subset\rset^{1+p+q}$ of positive
    length and containing $(\omega^\star,\chunk{a^\star}1p,\chunk{b^\star}2q)$ such that
    $\set{(\vartheta,\varphi_\star)\in\Theta}{\vartheta\in I^\star}\subset[\thv]$.
  \end{enumerate}
\end{theorem}
\begin{remark}\label{rem:standard-lodms} Let us briefly comment this result.
  \begin{enumerate}[label=(\arabic*)]
  \item The ergodicity of all the examples of
    \autoref{sec:standard-lodm} have been studied in the provided
    references and the parameter set $\Theta$ is always chosen to satisfy the
    assumptions of \autoref{thm:gener-ident-cond} in these references.
  \item\label{item:rem:standard-lodms1}
  If $p=q=1$, condition~\ref{item:lodm-ident-poly} is reduced to
  $b_1^\star\neq0$. Let us see what $b_1^\star=0$ would imply about the
  identifiability of the model in this simple case.  Taking $\tilde\psi^\theta$ as
  in~(\ref{eq:tilde-psi:affine:vector:case}) with $p=q=p'=q'=1$, if $\boldsymbol{\omega}(\theta)=\omega$,
  $A_1(\theta)=a^\star_1$ and $B_1(\theta)=b_1^\star=0$, then
  $\sequence{X}[k][\zsetp]$ is a deterministic sequence which, under the
  stationary distribution, has to be constantly equal to
  $x^\star=\frac{\omega^\star}{1-a_1^\star}$.
  But since the distribution of $\sequence{Y}$ is then uniquely defined by this
  constant, if one can find a parameter $\theta$ with corresponding
  coefficients $\omega,a_1,b_1$ such that $b_1=0$,
  $(\omega,a_1)\neq(\omega^\star,a_1^\star)$ yielding the same constant
  $\omega/(1-a_1)=\omega^\star/(1-a_1^\star)$, we see that the model is not identifiable.
\item Condition~\ref{item:lodm-ident-poly}
  holds for ``many'' parameters $\chunk{a^\star}{1}{p},\chunk{b^\star}{1}{q}$,
  e.g. for Lebesgue almost all ones in $\rset^{p+q}$.
\item The identifiability condition~\ref{item:lodm-ident-poly} is a
  well known sufficient condition in the standard GARCH$(p,q)$ models,
  see \cite[(A4)]{francq2004maximum} or
  \cite[Condition~(2.27)]{berkes03}. Assertion~\ref{item:thm:gener-ident-cond:cn}
  in \autoref{thm:gener-ident-cond} shows that it is also
  necessary at least for all parameters in the interior set of
  $\Theta$.
\item\label{item:exple-strict-inclusion} Suppose that $\Uset$ and
  $\Xset$ both contain at least two different points and take
  $\chunk {a^\star}1p$ and $\chunk {b^\star}1q$ to be all
  non-zero. Then it is easy to show for a standard LODM with known
  observation kernel that
  $\tilde{\Psi}^\theta_u(z)=\tilde{\Psi}^\thv_u(z)$ for all
  $u\in\Uset$ and $z\in\Zset$ implies $\theta=\thv$ and thus, we get
  that the left-hand side of~(\ref{eq:ident_fundamental}) reduces to
  the singleton $\{\thv\}$.  Since, as explained previously,
  \ref{item:lodm-ident-poly} is necessary to have $[\thv]=\{\thv\}$
  for all $\thv$ in the interior set of $\Theta$, we easily get
  examples for which the inclusion in~(\ref{eq:ident_fundamental}) is
  strict.
\item\label{item:identif:lodm:mpingarch} \autoref{thm:gener-ident-cond} can be applied to
  all the models mentioned in \autoref{sec:standard-lodm}. Let us
  examine the case of the MPINGARCH($p,q$) model of \cite{silva19},
  which constitutes a rich class of integer valued models. An
  MPINGARCH($p,q$) model is an LODM($p,q$) model with unknown
  observation kernel $G^{\varphi}(x,\cdot)$ defined as a mixed Poisson
  distribution with mean $x$ and variance proportional to
  $\varphi^{-1}$, and $\Upsilon(y)=y$. In \cite[Theorem~1]{silva19},
  the sufficient conditions for
  having~\ref{assum:gen:identif:unique:pi:gen}
  imply~\ref{item:vlodm-ident-cond-log-moment} (since
  $\mathbb{E}^\theta[|U_0]<\infty$)
  and~\ref{item:lodm-lipschitz-standard} (since they imply
  $\sum_ia_i<1$, with
  $a_i\geq0$). Conditions~\ref{item:lodm-lipschitz-standard-unknownG}
  and~\ref{item:lodm-degenerate-standard-unknown-case} also hold by
  definition of $G^{\varphi}(x,\cdot)$. Hence
  \autoref{thm:gener-ident-cond} applies and we get that
  \ref{item:lodm-ident-poly} is a sufficient condition for
  identifiablity. It is also necessary by
  Assertion~\ref{item:thm:gener-ident-cond:cn:unkown} of the theorem,
  at least for parameters in the interior of $\Theta$. This condition
  seems to be missing in \cite[Theorem~2]{silva19}.
\end{enumerate}
\end{remark}
\begin{proof}[Proof of \autoref{thm:gener-ident-cond}]
We only consider the case with unknown observation kernel (the case
with known observation kernel is obtained by removing the additional parameter $\varphi$).

Let $\thv=(\vartheta_\star,\varphi_\star)\in\Theta$ with
$\vartheta_\star=(\omega^\star,\chunk{a^\star}1p,\chunk{b^\star}1q)$.
As explained previously, the assumptions of
\autoref{thm:gener-ident-cond} are adapted from those derived in
\autoref{sec:vect-line-setting}. In particular, we have that
\ref{item:vlodm-ident-cond-cond-dens-def}--\ref{item:vlodm-ident-cond-log-moment}
hold. Hence \autoref{lem:ergo:ass:1} implies that
\ref{item:CCMShyp}--\ref{item:ident-nu-nondegenerate} hold in the
general setting with $\Uset=\rset$. Applying \autoref{thm:general-setting-main-result},
we get that \ref{item:X:from:Y:determ} holds and that
$$
\langle\thv\rangle=\set{\theta\in\Theta}{\tilde\Psi^\theta_u(z)=\tilde\Psi^{\thv}_u(z)\text{
    for all }(z,u)\in \mathrm{E}^\thv\times\Uset}\;,
$$
where $\langle\thv\rangle$ and $\mathrm{E}^\thv$ are as in
\autoref{def:ident:theta:gen:od:link} and \autoref{def:E-sets}.
Remember that \ref{item:ident:standard:G:unknown} says that
\ref{item:ident:G:gen:od:eq} holds with
$$
[\thv]_G=\set{\theta=(\vartheta,\varphi)\in\Theta}{\varphi=\varphi_{\star}}\;.
$$
By
\autoref{thm:the-main-result-general-setting}, we get that
the inclusion in~(\ref{eq:ident_fundamental_E}) is an equality and
$$
[\thv]=\set{\theta=(\vartheta,\varphi_\star)\in\Theta}{\tilde\Psi^\theta_u(z)=\tilde\Psi^{\thv}_u(z)\text{
    for all }(z,u)\in \mathrm{E}^\thv\times\Uset}\;.
$$
Note that we assumed that $0\in\Uset$ and
that~\ref{item:lodm-degenerate-standard-unknown-case} implies
$\Uset\neq\{0\}$, hence we can apply \autoref{lem:vlodm:angle-set}
which gives that $[\thv]$ is the set of all
$\theta=(\omega,\chunk{a}1p,\chunk{b}1q,\varphi_\star)\in\Theta$ such
that
  \begin{align*}
    \left(1-\sum_{k=1}^pa_k\right)\omega^\star
  &=\left(1-\sum_{k=1}^pa_k^\star\right)\omega   \;,  \\
    \text{and}\quad
    \frac{\sum_{k=0}^{q-1} b_{k+1}^\star\ z^{q-1-k}}
    {z^p  -\sum_{k=1}^pa^\star_{k}z^{p-k}}
 &=\frac{\sum_{k=0}^{q-1} b_{k+1}\ z^{q-1-k}}{z^p
  -\sum_{k=1}^pa_{k}z^{p-k}}
 \;.
  \end{align*}
Applying \autoref{lem:canonical-vlodm}, we easily get
Assertions~\ref{item:thm:gener-ident-cond:cs} and
\ref{item:thm:gener-ident-cond:cn:unkown}.
\end{proof}
\subsection{The bivariate example of \autoref{sec:bivariate-example}}
\autoref{thm:gener-ident-cond} can be extended to the standard VLODM
case of \autoref{def:standard-vlodms}. Here, for brevity, we do not
re-express the general VLODM assumptions
\ref{item:vlodm-ident-cond-cond-dens-def}-\ref{item:vlodm-ident-cond-log-moment}
in the \emph{standard} setting as we did previously for standard
LODMs. We only need to introduce the condition
\begin{addonehypprim}{SL}
\item\label{item:lodm-ident-poly-left-coprim} The polynomials $z^p-\sum_{k=1}^pA^{\star}_{k}z^{p-k}$ and
    $\sum_{k=0}^{q-1} B^{\star}_{k+1}\ z^{q-1-k}$ are left-coprime,
  \end{addonehypprim}
  which extends \ref{item:lodm-ident-poly} to the case $p',q'\geq1$.
The proof of the following result mimics the one of
\autoref{thm:gener-ident-cond} and is thus omitted.
\begin{theorem}
  \label{thm:gener-ident-cond-vlodm}
  Consider a \emph{standard} VLODM$(p,q,p',q')$
  satisfying~\ref{assum:gen:identif:unique:pi:gen} for some
  $p,q,p',q'\in\zsetpnz$. Suppose that $0\in\Uset$ and that
  \ref{item:vlodm-ident-cond-cond-dens-def}-\ref{item:vlodm-ident-cond-log-moment}
  and~\ref{item:ident:standard:G:unknown} hold.  Then, for all
  $\thv=(\boldsymbol{\omega}^\star,\chunk{A^\star}1p,\chunk{B^\star}1q,\varphi_\star)\in\Theta$
  , the inclusion in~(\ref{eq:ident_fundamental_E}) is an equality and the
  two following assertions hold.
  \begin{enumerate}[label=(\roman*)]
  \item\label{item:thm:gener-ident-cond-vlodm-standard:cs}
    Condition~\ref{item:lodm-ident-poly-left-coprim} implies that $[\thv]$ reduces to the singleton $\{\thv\}$.
  \item\label{item:thm:gener-ident-cond-vlodm-standard:cn:unkown} If~\ref{item:lodm-ident-poly-left-coprim} does
    not hold, then there exists an open segment $I^\star\subset\rset^{1+p+q}$ of positive
    length and containing $(\omega^\star,\chunk{a^\star}1p,\chunk{b^\star}2q)$ such that
    $\set{(\vartheta,\varphi_\star)\in\Theta}{\vartheta\in I^\star}\subset[\thv]$.
  \end{enumerate}
\end{theorem}
\begin{remark}\label{rem:standard-vlodms-known}
  In \autoref{thm:gener-ident-cond-vlodm}, for brevity, we only
  stated the case with unknown observation kernel. The case with
  known observation kernel follows by removing the
  parameters $\varphi$ and  $\varphi_\star$ in the
  statement and by replacing \ref{item:ident:standard:G:unknown} by \ref{item:ident:standard:G}.
\end{remark}
Since the bivariate integer valued GARCH model of
\autoref{sec:bivariate-example} is a \emph{standard} VLODM($1,1,2,2$)
with unknown observation kernel, we just need to check the assumptions
of \autoref{thm:gener-ident-cond-vlodm}. Ergodicity (hence our
Assumption~\ref{assum:gen:identif:unique:pi:gen}) is stated in
\cite[Theorem~1]{cui-zhu-2018} under their set of condition (a) on the
parameter $\theta=(\boldsymbol{\omega},A,B,\varphi)$. Their conditions
for ergodicity implies some operator norm of $A$ to be strictly less
than 1, which implies $\mathrm{I}_{2}-A\;z$ to be invertible for
$|z|\leq1$ and thus~\ref{item:vlodm-ident-cond-cond-dens-def}
holds. Since for all $x_1,x_2>0$, the bivariate distribution defined
by~(\ref{eq:bivariate-poisson}) has positive probability on all points
$(y_1,y_2)\in\zsetp^2$, it cannot have probability one on a line of
$\rset^2$, hence~\ref{item:vlodm--non-degenerateident-G}. Also it is
claimed following \cite[Theorem~1]{cui-zhu-2018} that
$\mathbb{E}^\theta[U_0]$ is well defined and thus
\ref{item:vlodm-ident-cond-log-moment} holds. Applying
\autoref{thm:gener-ident-cond-vlodm}, we get that, for any interior
point $\thv$ of the parameter space,
Condition~\ref{item:lodm-ident-poly-left-coprim} is a necessary and
sufficient condition to have identifiability of $\thv$. In the
bivariate integer valued GARCH model (for which $p=q=1$), this
condition reads for
$\thv=(\boldsymbol{\omega}^\star,A^\star,B^\star,\varphi^\star)$ as
$\mathrm{I}_{2}\,z-A^\star$ and $B^\star$ to be left coprime. If $\thv$
does not satisfy this condition, consistent estimation of $\thv$ is
not possible. Hence we believe that this assumption is missing in
\cite[Theorem~2]{cui-zhu-2018}. A precise counter-example is for
instance obtained by setting
$$
A^\star=\alpha
\begin{bmatrix}
  1&1\\
  1&1
\end{bmatrix}
\quad\text{and}\quad
B^\star=\beta
\begin{bmatrix}
  1&1\\
  1&1
\end{bmatrix}\;,
$$
where $\alpha,\beta>0$ are arbitrary (and chosen in order to make
$\thv=(\boldsymbol{\omega}^\star,A^\star,B^\star,\varphi^\star)$ in
the interior of $\Theta$). One can show that
$\mathrm{I}_{2}\,z-A^\star$ and $B^\star$ are not left coprime since they
both admit the same non-unimodular left divisor
$L=
\begin{bmatrix}
  (z-\alpha)& 1\\
  (-\alpha)& 1
\end{bmatrix}$
as shown by the following identities:
$$
\mathrm{I}_{2}\,z-A^\star=L\;
\begin{bmatrix}
  1&-1\\
  0&(z-2\alpha)
\end{bmatrix}
\quad\text{and}\quad
B^\star=L\;
\begin{bmatrix}
  0&0\\\beta&\beta
\end{bmatrix}
\;.
$$

\subsection{The non-linear GARCH of \autoref{sec:non-linear-garch}}
Consider \ref{item:non-linear-case1} of
\autoref{sec:non-linear-garch}, for which $\delta$ is not
included in the set of parameters. Then the
non-linear GARCH model is a standard VLODM($p,q,1,2$) model and
identifiability can be treated using \autoref{thm:gener-ident-cond-vlodm}. Using that
$\Upsilon(y)=((y^+)^\delta,(y^-)^\delta)$, for all $x>0$, $\tilde
G(x,\cdot)$ in~(\ref{eq:def:tildeG}) (we omit $\theta$ as $G$ does not
depend on $\theta$ here) has support included in
$\rsetp\times\{0\}\cup\{0\}\times\rsetp$, where it is defined, for all
Borel set $A\subset\rsetp$ by
$$
\tilde{G}(x;\{0\}\times A)=\PP(x\,(\eta^+)^\delta\in A)
\quad\text{and}\quad\tilde{G}(x; A\times\{0\})=\PP(x\,(\eta^-)^\delta\in A)
$$
It follows that our Assumption~\ref{item:vlodm--non-degenerateident-G}
is equivalent to having that $0<\PP(\eta>0)<1$ and that there is no
pair $\{u,v\} $, $u\neq v\in\rset$, such that $\PP(\eta\in\{u,v\})=1$,
which is exactly the condition appearing in the second part of
\cite[A3]{hamadeh-zakoian:jspi2011}. Our
conditions~\ref{item:vlodm-ident-cond-cond-dens-def} (which here,
since $A_k=a_k\geq0$ for all $k=1,\dots,p$
simply reads $\sum_ka_k<1$) and \ref{item:vlodm-ident-cond-log-moment}
are usual byproducts of showing the ergodicity
condition~\ref{assum:gen:identif:unique:pi:gen}, see
\cite[Appendix~A]{hamadeh-zakoian:jspi2011}.
One can thus apply our \autoref{thm:gener-ident-cond-vlodm} (in its know
observation kernel version, see
\autoref{rem:standard-vlodms-known}) and
obtain the necessary and sufficient condition
\ref{item:lodm-ident-poly-left-coprim} which in the case where the
parameters $A_k$ are denoted by $a_k\in\rset$ for $k=1,\dots,p$ and the parameters
$B_k$ denoted by
$\mathbf{b}_k\in\rset^2$ for $k=1,\dots,q$, becomes, for any
$\thv=(\omega^\star,\chunk{a^\star}1p,\chunk{\mathbf{b}^\star}1q)$
with
$\mathbf{b}^\star_k=\begin{bmatrix}b^\star_k(1)&b^\star_k(2)\end{bmatrix}$
for $k=1,\dots,q$,
\begin{hyp}{NLG}
\item The polynomial
  $z^p-\sum_{k=1}^pa^{\star}_{k}z^{p-k}$     have no common complex
  roots neither with the polynomial
    $\sum_{k=0}^{q-1} b^{\star}_{k+1}(1)\ z^{q-1-k}$
nor with the polynomial     $\sum_{k=0}^{q-1} b^{\star}_{k+1}(2)\ z^{q-1-k}$.
\end{hyp}
This condition is similar to that appearing in the identifiability
condition \cite[A4]{hamadeh-zakoian:jspi2011} used in a mis-specified
context. Our result shows that this condition is necessary in the
interior of the parameter set in the well-specified case, and remains
valid for much larger choices of observation kernels.

\subsection{The SETPAR model of \autoref{sec:other-non-linear}}\label{sec:setpar-appli}
We have the following result for the self-excited threshold Poisson autoregression
model.
\begin{theorem}\label{thm:setpar}
Consider the SETPAR
model introduced in \autoref{sec:other-non-linear}. Let
$$
\Theta\subset\set{(\omega_1,\omega_2,a_1,a_2,b_1,b_2,r)\in\rsetpnz^2\times[0,1)^2\times\rsetp\times[0,1)\times\zsetp}{a_2+b_2<1}\;,
$$
Then ~\ref{assum:gen:identif:unique:pi:gen} holds. Let
$\thv=(\omega_1^\star,\omega_2^\star,a_1^\star,a_2^\star,b_1^\star,b_2^\star,r^\star)\in\Theta$
satisfy at least one of the two following conditions.
\begin{enumerate}[label=(\roman*)]
\item\label{item:setpar-cond1} $b_1^\star>0$ and $r^\star\geq1$;
\item\label{item:setpar-cond2} $b_2^\star>0$.
\end{enumerate}
Then we have
\begin{equation}
  \label{eq:ident_fundamental:tpa}
[\thv]=\set{\theta\in\Theta}{\psi^\theta_{y}(x)=\psi^\thv_{y}(x)\quad\text{for
  all}\quad x\in\rset\text{ and }y\in\zsetp}\;,
\end{equation}
where $\psi^\theta_{y}(x)$ is defined by~(\ref{eq:tpar-link}).
\end{theorem}
\begin{remark}
  Let us briefly comment this result.
  \begin{enumerate}[label=(\arabic*)]
  \item The case where neither~\ref{item:setpar-cond1}
    nor~\ref{item:setpar-cond2} hold ($b_1^\star=b_2^\star=0$ or $r^\star=b_2^\star=0$) is somehow degenerate,
    similarly to the non-threshold case mentioned
    in \autoref{rem:standard-lodms}\ref{item:rem:standard-lodms1}. We
    think it should be treated separately but we omit this very
    special case here for brevity.
  \item As explained after~(\ref{eq:ident_fundamental}), the
    identity~(\ref{eq:ident_fundamental:tpa}) is the best we could hope
    for this model since the distribution $\tilde\PP^\theta$ of the observations is entirely
    determined by the mapping $(u,x)\mapsto\psi^\theta_u(x)$ on
    $\zsetp\times\rsetp$.
  \item The identity~(\ref{eq:ident_fundamental:tpa}) shows in
    particular that $\thv$ is not identifiable if $r^\star=0$ (since
    changing $b^\star_1$ will have no effect on the mapping  $(u,x)\mapsto\psi^\theta_u(x)$ on
    $\zsetp\times\rsetp$). Another case of non-identifiability is when $a_1^\star=a_2^\star$ and $\omega_1^\star + b_1^\star (r+1)=\omega_2^\star + b_2^\star (r+1)$. In such a case, we have for all $x\in \rset$,
    $$
  \omega_1^\star +  a_1^\star x+ b_1^\star y
=  \omega_2^\star +  a_2^\star x+ b_2^\star y\quad\text{at $y=r+1$.}
$$
Then, setting
$\theta=(\omega_1^\star,\omega_2^\star,a_1^\star,a_2^\star,b_1^\star,b_2^\star,r^\star+1)$,
we immediately have that $\psi^\theta_{y}(x)=\psi^\thv_{y}(x)$ for all
$x\in\rset$ and $y\in\zsetp$. In particular consistent estimation of
$\thv$ as claimed in \cite[Theorem~2]{wang2014self} is not
possible for such a parameter $\thv$.
  \end{enumerate}
\end{remark}
\begin{proof}[Proof of \autoref{thm:setpar}]
  A natural choice for $\Xset$ is $\rsetpnz$ but in order to meet
  Assumption \ref{item:CCMShyp} with $\Xmet(x,x')=|x-x'|$ we take
  $\Xset=\rsetp$ with $G^\theta(0,\cdot)$ arbitrarily set to be
  Bernoulli with mean $1/2$ for convenience (it actually has no
  influence on $\PP^\theta$ since $\omega_0,\omega_1>0$ in the
  condition on $\Theta$). We set $\Upsilon(y)=y$ so that $U_k=Y_k$ for
  all $k$ and the reduced link function $\tilde\psi$ is the same as
  the non-reduced one. Moreover since $p=q=1$, we are in the case
  where $Z_k=X_k$ for all $k$, $\zset=\Xset=\rsetp$ and
  $\tilde{\Psi}_u^\theta=\tilde\psi_u^\theta$ for all $u$. By
  \cite[Theorem~1]{wang2014self}, with $\Theta$ satisfying the given
  condition, Assumption~\ref{assum:gen:identif:unique:pi:gen} holds,
  and, moreover, for any $\ell>0$ and $\theta\in\Theta$,
  $\mathbb{E}^\theta[U_0^\ell]<\infty$ (in fact, on can prove that,
  for any $\theta\in\Theta$, there exists $\ell>0$ such that
  $\mathbb{E}^\theta[\exp(\ell U_0)]<\infty$). This moment condition
  implies that the log moment condition
  \ref{item:ident-cond-log-moment} holds for any
  $\initmlex_1\in\Xset$.  Clearly, we have
  $\mathrm{Lip}^\theta_1=a_1\vee a_2$ and, since $p=q=1$,
  $\mathrm{Lip}^\theta_n\leq(\mathrm{Lip}^\theta_1)^n$ for all
  $n\geq1$. Thus the above condition on $\Theta$ also
  implies~\ref{assum:bound:rho:gen}. As for
  \ref{item:continuity-in-x-for-u-given}, it trivially holds (since
  $y=u$ is fixed in this condition). Hence with
  \autoref{lem:Lip:XY-cond}, we get that
  \ref{assum:gen:identif:unique:pi:gen}--\ref{item:continuity-in-x-for-u-given}
  holds, with
  definitions~(\ref{eq:definition:fF:asalimit:with:domain}) for
  checking~\ref{item:X:from:Y:determ}.
  Assumption~\ref{item:ident:G:gen:od:eq} is also immediate with
  $[\theta]_G=\Theta$ and \autoref{thm:strong:identifiability:gen:od}
  gives that, for any $\thv\in\Theta$,
$$
[\thv]=\langle\thv\rangle\;,
$$
where $\langle\thv\rangle$ can be defined
by~(\ref{eq:to-prove-deidentif:def-angle-set-limit-case}).  Assumption
\ref{item:ident-nu-nondegenerate} holds by
\autoref{rem:ident-non-degenerate}\ref{item:rem:ident-non-degenerate1}. Hence
all the assumptions of \autoref{thm:general-setting-main-result} hold.
Take now
$\thv=(\omega_1^\star,\omega_2^\star,a_1^\star,a_2^\star,b_1^\star,b_2^\star,r^\star)\in\Theta$
satisfying~\ref{item:setpar-cond1} or~\ref{item:setpar-cond2}. Let
$\theta\in\Theta$, $u\in\Uset$ and $z\in\Zset$.  To conclude the
proof, it is now sufficient to check that if $\tilde \Psi^\theta_u$
and $\tilde \Psi^\thv_u$ coincide on
$\set{\tF[\thv]{v}(z)}{n\in\zsetp,\,v\in\Uset^n}$, then they must
coincide on $\Zset$, so that we can apply
\autoref{thm:the-main-result-general-setting}. Observe that by
definition of the link function in~(\ref{eq:tpar-link}),
if~\ref{item:setpar-cond1} or~\ref{item:setpar-cond2} holds, then
$v\mapsto\psi_v^\thv(z)$ takes at least two different
values on $\Uset=\zsetp$. Since $\tF[\thv]{v}(z)=\psi_v^\thv(z)$ for
all $v\in\Uset$, these two different values belong to
$\set{\tF[\thv]{v}(z)}{n\in\zsetp,\,v\in\Uset^n}$. Now, since
$\tilde \Psi^\thv_u=\psi_u^\thv$ and $\tilde
\Psi^\theta_u=\psi_u^\theta$ are affine functions, if they coincide in
two points they must coincide everywhere, and the proof is concluded.
\end{proof}
\subsection{The PARX model of \autoref{sec:parx-model}}\label{sec:parx-model-appli}
We have the following result for the Poisson autoregression model with
exogenous covariates.
\begin{theorem}\label{thm:parx}
Consider  the PARX model defined in \autoref{sec:parx-model}, which is a
VLODMX($p,q,1,1+d$). Suppose
that~\ref{assum:gen:identif:unique:pi:gen:X},
\ref{item:vlodm-ident-cond-cond-dens-def} and
\ref{item:vlodm-ident-cond-log-moment} hold, and that the exogenous
kernel $H$ satisfies the following.
\begin{hyp}{P}
  \item\label{item:PARX--non-degenerateident-G-X} We have
    $H(v;\{\chunk f1d(\cdot)\in A\})<1$ for all  $v\in\Vset$ and affine hyperplanes
  $A\subset\rset^{d}$.
\end{hyp}
Then, for all $\thv\in\Theta$, the equivalent class $[\thv]$ reduces to the
singleton $\{\thv\}$.
\end{theorem}
\begin{remark}
  Let us briefly comment this result.
  \begin{enumerate}[label=(\arabic*)]
  \item \ref{item:PARX--non-degenerateident-G-X} is a
    natural assumption as it basically says that the covariates
    $f_1(V_{k}),\dots,f_d(V_{k})$ are not linearly related
    conditionally to $V_{k-1 }$. If they were, it would suggest using
    a smaller set of covariates.
  \item The ergodicity of PARX models (our
    assumption~\ref{assum:gen:identif:unique:pi:gen:X} is treated in
    \cite[Theorem~1]{AGOSTO2016640} under some assumption on the
    covariate kernel $H$ (in their assumption~2). Their assumption~3
    used for proving \ref{assum:gen:identif:unique:pi:gen:X} implies
    $\sum_k a_k<1$ which implies our
    assumption~\ref{item:vlodm-ident-cond-cond-dens-def} (since
    $a_k\geq0$ for all $k=1,\dots,p$). Note also that
    \cite[Theorem~1]{AGOSTO2016640} implies that
    $\mathbb{E}^\theta[|U_0|]<\infty$ and thus our
    assumption~\ref{item:vlodm-ident-cond-log-moment}. On the other
    hand, their identifiability condition
    \cite[Assumption~5]{AGOSTO2016640} include a condition on
    parameters $\chunk{a^\star}1p,\chunk{b^\star}1q$ similar to our
    condition ~\ref{item:lodm-ident-poly} used for standard LODM's in
    \autoref{thm:gener-ident-cond} above. \autoref{thm:parx} shows
    that such a condition can in fact be dropped in case of exogenous
    covariates provided that the mild condition
    \ref{item:PARX--non-degenerateident-G-X} holds. 
  \item This result is similar to  \autoref{thm:gener-ident-cond} for
    the standard LODM. It is of interest to note that
    there is no additional condition on $\chunk{\gamma^\star}1d$ for
    identifiability.
  \item \autoref{thm:parx} easily extends to more general observation
    kernels (known or unknown) $G^\theta$, provided that, as in
    \autoref{thm:gener-ident-cond},
    assumptions~\ref{item:lodm-degenerate-standard}-\ref{item:ident:standard:G} hold
    if the observation kernel is known, or
    \ref{item:lodm-degenerate-standard-unknown-case}-\ref{item:ident:standard:G:unknown}
    if it is unknown. However, the ergodicity would require a
    specific treatment in these cases, as only the Poisson case has
    been considered up to our knowledge.
  \end{enumerate}
\end{remark}
\begin{proof}[Proof of \autoref{thm:parx}]
  For the PARX model we set
  $\Upsilon(y,v)=(y,f_1(v),\dots,f_d(v))\in\Uset=\rset^{1+d}$ and
  $\Xset=\rsetp$ (with $G(0,\cdot)$ arbitrarily set, say, to be
  Bernoulli with mean $1/2$, as in the proof of \autoref{thm:setpar}) and
  \ref{item:CCMShyp} holds with $\Xmet(x,x')=|x-x'|$ and
  $\Umet(u,u')=|u-u'|$ where $|\cdot|$ here is an arbitrary norm on
  $\rset^{1+d}$. In this case, we have that, for all
  $(x,v)\in\Xset\times\Vset$, if $(Y,W)\sim\tilde{G}((x,v),\cdot)$
  with $Y$ valued in $\rset$, $W$ valued in $\rset^d$ and $\tilde{G}$
  defined by~(\ref{eq:def:tildeG:X}), we have that $Y$ and $W$ are
  independent and $Y$ follows a Poisson distribution. Thus
  \ref{item:vlodm--non-degenerateident-G-X} is equivalent to
  \ref{item:PARX--non-degenerateident-G-X}. We can thus apply
  \autoref{lem:ergo:ass:1} and get that
  Assumptions~\ref{item:CCMShyp}, \ref{assum:bound:rho:gen},
  \ref{item:ident-cond-log-moment},
  \ref{item:continuity-in-x-for-u-given} and
  \ref{item:ident-nu-nondegenerate-X} hold with
  $\Uset=\rset^{1+d}$. Then, having assumed
  \ref{assum:gen:identif:unique:pi:gen:X}, we can apply
  \autoref{thm:general-setting-main-result} and get that
  \ref{item:X:from:Y:determ} holds as well as the
  identity~(\ref{eq:thm-ident-general-odm}).
  Assumption é\ref{item:ident:G:gen:od:eq} is immediate with
$[\thv]_G=\Theta$ and,
Applying \autoref{thm:strong:identifiability:gen:od} and the previous display
we get that
$$
[\thv]=\set{\theta\in\Theta}{\tilde\Psi^\theta_u(z)=\tilde\Psi^{\thv}_u(z)\text{
    for all }(z,u)\in \mathrm{E}^\thv\times\Uset}\;.
$$
where $\langle\thv\rangle$ and $\mathrm{E}^\thv$ are as in
\autoref{def:ident:theta:gen:od:link} and \autoref{def:E-sets}.
\autoref{lem:vlodm:angle-set} and the definitions of $\chunk A1p$ and
$\chunk B1q$ in \autoref{sec:parx-model} now give that $[\thv]$ is the
set of all $\theta=(\omega,\chunk{a}1p,\chunk{b}1q,\chunk \gamma1d)\in\Theta$ such that
  \begin{align*}
    \left(1-\sum_{k=1}^pa_k\right)\omega^\star
  &=\left(1-\sum_{k=1}^pa_k^\star\right)\omega   \;,  \\
    \frac{\sum_{k=0}^{q-1} b_{k+1}^\star\ z^{q-1-k}}
    {z^p  -\sum_{k=1}^pa^\star_{k}z^{p-k}}
 &=\frac{\sum_{k=0}^{q-1} b_{k+1}\ z^{q-1-k}}{z^p
   -\sum_{k=1}^pa_{k}z^{p-k}}   \;,  \\
    \left(z^p  -\sum_{k=1}^pa^\star_{k}z^{p-k}\right)^{-1}\;d_i^\star\, z^{q-1}
    &=
          \left(z^p -\sum_{k=1}^pa_{k}z^{p-k}\right)^{-1} \;d_i\,
      z^{q-1}\quad\text{for }1\leq i\leq d
 \;.
  \end{align*}
  Note that the last line in this set of equations is equivalent to
    \begin{align*}
 a^\star_{k}=a_k\quad\text{and}\quad     d_i=d^\star_i\quad\text{for
      all }1\leq i\leq d \quad\text{and}\quad1\leq k\leq p\;.
    \end{align*}
    But the the two first lines of the previous display give that
    $\omega=\omega^\star$ and $\chunk{b^\star}1q=\chunk b1q$, thus
    $\theta=\thv$ and the
    proof is concluded.
\end{proof}

\section{Postponed Proofs}\label{sec:append:proofs:gen:od}
\subsection{Proof of  \autoref{lem:Lip:XY-cond}}
  \label{sec:proof-XY-cond}
We first derive the following result.
\begin{lemma}\label{lem:Lip:XY-cond:geom}
  \ref{assum:bound:rho:gen} implies that for all $\theta\in\Theta$, there
  exist $C>0$ and $\rho\in(0,1)$ such that   $\mathrm{Lip}_n^\theta \leq C\ \rho^n$ for
  all $n\in\zsetpnz$.
\end{lemma}
\begin{proof}

By~(\ref{eq:Lip:constant:n:def}),~(\ref{eq:def:Zmet})
  and~(\ref{eq:psi:Psi:relation:recip}), we have, for all $n\in\zsetpnz$, using the
  convention $\mathrm{Lip}^\theta_{m}=1$ for $m\leq0$,
  \begin{equation}
    \label{eq:lipschitz:Psi}
  \sup_{u\in\Uset^{n},v\in\Zset^2}\frac{\Zmet\circ\tF{u}^{\otimes2}(v)}{\Zmet(v)}\leq
\1_{\{n< q\}}\vee\left(  \max_{0\leq j< p}\mathrm{Lip}_{n-j}^\theta\right)\;.
  \end{equation}
   Hence \ref{assum:bound:rho:gen} implies that there exists $m\geq1$ and $L\in(0,1)$ such that,
   for all $u\in\Uset^{m+1}$, $\tF{u}$ is $L$-Lipschitz. Now observe that,
   by~(\ref{eq:psi:Psi:relation}), for all $n=km+r$ with $k\geq0$ and $0\leq
   r<m$, for all $u=\chunk u{-n}0\in\Uset^{n+1}$, we can write $\tF{u}$ as
$$
\tF{\chunk u{(1-m)}0}\circ\tF{\chunk u{(1-2m)}{(-m)}}\circ\dots\circ\tF{\chunk
  u{(1-km)}{(-(k-1)m)}}\circ \tF{\chunk u{-n}{(-km)}}\;,
$$
and in this composition, the $k$ first functions are $L$ Lipschitz and the last
one is $L'=1\vee\max\set{\mathrm{Lip}^\theta_{j}}{0< j\leq m}$-Lipschitz. Hence, for all $z,z'\in\Zset$,
$$
\Xmet(\tf{u}(z),\tf{u}(z'))\leq\Zmet(\tF{u}(z),\tF{u}(z'))\leq L'\ L^k\,\Zmet(z,z')\;.
$$
Hence the result by setting $\rho=L^{1/m}\in(0,1)$.
 \end{proof}
 We can now prove \autoref{lem:Lip:XY-cond}.  Let $\theta, \thv \in \Theta$ and let $\initmle\in\Zset$. Denote for all $n\in\zsetp$,
 $$
 X^{(n)}=\tf{\chunk U{(-n)}0}(\initmle)\,.
 $$
 Then,
 by~(\ref{eq:notationItere:f:cl:gen}) we have, for all $n\in\zsetpnz$,
  $$
  X^{(n)}=\tf{\chunk U{-n+1}0}\circ
  \tilde\Psi^\theta_{U_{-n}}(\initmle)\;,
  $$
  and, by~(\ref{eq:Lip:constant:n:def}), we get
  \begin{equation}
    \label{eq:cauchy-lem:lip}
  \Xmet\left(  X^{(n)},  X^{(n-1)}\right)\leq \mathrm{Lip}_{n}^\theta\, \Zmet\left(\initmle,\tilde\Psi^\theta_{U_{-n}}(\initmle)\right)\;.
  \end{equation}
   Using~(\ref{eq:def:Zmet}) with $\initmle=(\initmlex,\initmleu)$ and
   $\initmlex_1=\dots=\initmlex_p$ and
   $\initmleu_1=\dots=\initmleu_{q-1}$
   we get that
  $$
  \Zmet\left(\initmle,\tilde\Psi^\theta_{U_{-n}}(\initmle)\right)=
  \Xmet\left(\initmlex_1,\tilde\psi^\theta_{(\initmleu,U_{-n})}(\initmlex)\right)\bigvee\Umet(\initmleu_1,U_{-n})
  $$
  Hence, for all
  $\alpha>0$, Condition~(\ref{eq:cond:moment:XY}) implies as $n\to\infty$,
  $$
  \Zmet\left(\initmle,\tilde\Psi^\theta_{U_{-n}}(\initmle)\right)=
O(\rme^{\alpha n})\qquad \tilde\PP^\thv\as
  $$
  The last display with~(\ref{eq:cauchy-lem:lip}) and
  \autoref{lem:Lip:XY-cond:geom} gives that, $\tilde\PP^\thv\as$,
  $\nsequence{X^{(n)}}{n \in \zsetp}$ is a Cauchy sequence, hence converges in
  $\Xset$. Therefore, $\chunk U{(-\infty)}{0}\in \mathrm{D}^\theta$,
  $\tilde \PP^\thv\as$ By~(\ref{eq:def:gen-ob})
  or~(\ref{eq:def:gen-ob-X}) depending whether an ODM or an ODMX is considered, we also have, under $\PP^\theta$, for all
  $n\in\zsetp$, $X_1=\tf{\chunk U{-n}0}(Z_{-n})$. Thus,~(\ref{eq:Lip:constant:n:def})
  also implies
  $$
  \Xmet(X_1,X^{(n)})\leq  \mathrm{Lip}^\theta_{n+1}\,
  \Zmet\left(Z_{-n},\initmle\right)\qquad\PP^\theta\as
  $$
  By stationarity, $\Zmet(Z_{-n},\initmle)$ is bounded in probability
  under $\PP^\theta$, hence $X^{(n)}$ converges to $X_1$ in
  probability if \ref{assum:bound:rho:gen} holds. We thus
  obtain~(\ref{eq:identif-XY}), and since this holds for all
  $\theta\in\Theta$, Assumption~\ref{item:X:from:Y:determ} holds.

  Let us now check~(\ref{eq:ident_fundamental_E}). Take
  $\thv,\theta\in\Theta$ and suppose that $\theta$ belongs to the set
  in left-hand side of the inclusion~(\ref{eq:ident_fundamental_E}),
  that is, $G^{\theta}=G^{\thv}$ and
  $\tilde{\Psi}^{\theta}_u(z)=\tilde{\Psi}^{\thv}_u(z)$ for all
  $(z,u)\in\mathrm{E}^\thv\times\Uset$. By \autoref{rem:E-theta-Zk} we
  have $Z_k\in\mathrm{E}^\thv$, $\PP^{\thv}\as$ Thus we get that
  $\tilde{\Psi}^{\theta}_{U_k}(Z_k)=\tilde{\Psi}^{\thv}_{U_k}(Z_k)$ 
   $\PP^{\thv}\as$ and with \autoref{rem:def:gen-ob-Z}, we obtain
   that,    $\PP^{\thv}\as$,
   $(Y_k,X_k)$ (resp. $(Y_k,X_k,V_k)$) satisfy the iterative
   equations~(\ref{eq:def:gen-ob}) (resp.~(\ref{eq:def:gen-ob-X})), and
   by~\ref{assum:gen:identif:unique:pi:gen}
   (resp. \ref{assum:gen:identif:unique:pi:gen:X}), we conclude that
   $\PP^\theta=\PP^\thv$. Hence $\theta\in[\thv]$
   and~(\ref{eq:ident_fundamental_E}) is proved.

    Finally, we check that~(\ref{eq:identif-XYa-thv}) holds when $\tilde \psi^\theta_u$
  is continuous for all $u\in\Uset^q$. Since we have shown that  $\chunk U{(-\infty)}{0}\in \mathrm{D}^\theta$, $\tilde\PP^{\thv}\as$ and $\tilde\PP^{\thv}$ is shift invariant, we
  have, for all $k\in\zset$,
  $$
  \tf[\theta]{\chunk U{(-\infty)}{k}}=\lim_{n\to\infty}\tf{\chunk
    U{-n}k}(\initmle)\quad\tilde\PP^{\thv}\as\;.
  $$
  Observe that, for all $n\geq p\vee q$,
  we have, for all $\chunk u{(-n)}0\in\Uset^{n+1}$,
  $$
  \tf{\chunk u{(-n)}0}(\initmle)=\tilde \psi^\theta_{\chunk u{(-q+1)}0}\left(\setvect{\tf{\chunk
      u{(-n)}{(-j)}}(\initmle)}{-p\leq j\leq -1}\right)\;.
  $$
  By continuity of $\tilde \psi^\theta_{u}$ and using the previous display, we can
  take the limit as $n\to\infty$  under $\tilde\PP^{\thv}$ and
  obtain~(\ref{eq:identif-XYa-thv}).
\subsection{Proof of \autoref{thm:strong:identifiability:gen:od}}
\label{sec:proof-converse-inclusion}
First observe that~(\ref{eq:identif-XY}) implies for all $\theta \in \Theta$,
  \begin{equation}
    \label{eq:cond-prob-XY-psi}
\tCPPu[\theta]{Y_1\in\cdot}{\chunk Y{(-\infty)}0}=G^\theta\left(\tf[\theta]{\chunk{U}{(-\infty)}{0}};\cdot\right)\qquad\tilde{\PP}^\theta\as
  \end{equation}
  Let us now show that any $\theta\in[\thv]$ belongs to
  $[\thv]_G\cap\langle\thv\rangle$, that is, $\theta\in[\thv]_G$,
  and~(\ref{eq:to-prove-deidentif:gen:od}) and~(\ref{eq:identif-XYa-thv}) hold
  true. Since $\tilde\PP^\theta=\tilde\PP^\thv$,~(\ref{eq:cond-prob-XY-psi}),
  which also holds with $\theta$ replaced by $\thv$, yields
$$
G^\theta\left(\tf[\theta]{\chunk{U}{(-\infty)}{0}}; \cdot
\right)=G^\thv\left(\tf[\thv]{\chunk{U}{(-\infty)}{0}}; \cdot \right)\qquad{\tilde{\PP}^\thv}\as
$$
By~\ref{item:ident:G:gen:od:eq}, we obtain that
$\theta\in[\thv]_G$ and~(\ref{eq:to-prove-deidentif:gen:od}) holds.
By \autoref{lem:identif-XY},~(\ref{eq:identif-XY}) implies~(\ref{eq:identif-XYa}), and using
$\tilde\PP^\theta=\tilde\PP^\thv$, we obtain~(\ref{eq:identif-XYa-thv}).
Thus $\theta\in\langle\thv\rangle$.

It remains to show that
$[\thv]_G\cap\langle\thv\rangle\subseteq[\thv]$. We prove this
inclusion in the case of an ODMX
satisfying~\ref{assum:gen:identif:unique:pi:gen:X}. (The case of an
ODM is readily obtained by removing the variables $V_k$'s in the
reasoning).  Let $\theta\in[\thv]_G$ such
that~(\ref{eq:to-prove-deidentif:gen:od})
and~(\ref{eq:identif-XYa-thv}) hold true. Since~(\ref{eq:identif-XY})
holds with $\theta$ replaced by
$\thv$,~(\ref{eq:to-prove-deidentif:gen:od}) gives that
$X_1=\tf[\theta]{\chunk U{(-\infty)}0}$ $\PP^{\thv}\as$ Since
$\PP^{\thv}$ is shift invariant, we get, for all $k\in\zset$,
$X_{k+1}=\tf[\theta]{\chunk U{(-\infty)}k}$ $\PP^{\thv}\as$
With~(\ref{eq:identif-XYa-thv}), we obtain
$X_1=\tilde\psi^\theta_{\chunk U{(-q+1)}{0}}\left(\chunk
  X{(-p+1)}{0}\right)$ $\PP^\thv\as$ Since $\PP^\thv$ is shift
invariant, we thus have, for all $k\in\zset$,
  \begin{equation}
    \label{eq:link-XY-transfer}
X_{k+1}=\tilde\psi^\theta_{\chunk
   U{(k+1-q)}{k}}\left(\chunk X{(k+1-p)}{k}\right)
     \quad\PP^\thv\as
   \end{equation}
   On the other hand, by definition of $\PP^\thv$ and
using~\ref{item:ident:G:gen:od:eq} with $\theta\in[\thv]_G$, we have that
$$
\CPPu[\thv]{Y_1\in\cdot}{\chunk X{(-\infty)}1,\chunk Y{(-\infty)}0,\chunk V{(-\infty)}0}=
G^\thv(X_1;\cdot)=G^\theta(X_1;\cdot)\qquad\PP^\thv\as
$$
And using again that $\PP^\thv$ is shift-invariant, for all $k\in\zset$,
$$
\CPPu[\thv]{Y_{k}\in\cdot}{\chunk X{(-\infty)}k,\chunk Y{(-\infty)}{(k-1)},\chunk V{(-\infty)}{(k-1)}}=
G^\theta(X_k;\cdot)\qquad\PP^\thv\as
$$
This, with~(\ref{eq:link-XY-transfer}), shows that $\PP^\thv$ is a
shift-invariant solution
of~(\ref{eq:def:gen-ob-X}). By~\ref{assum:gen:identif:unique:pi:gen:X}, we conclude
that $\PP^\thv=\PP^\theta$, and thus $\theta\in[\thv]$.

\subsection{Proof of \autoref{lem:ergo:ass:1}}
\label{sec:proof-lem:ergo:ass:1}

\noindent\textbf{Assertion~\ref{item:item:CCMShyp}} is obvious.

\noindent\textbf{Proof of Assertion~\ref{item:assum:bound:rho:gen}.}
In the  vector linear setting we have, for all $n\in\zsetpnz$ and all $(z,z',u)\in\Zset^{2}\times\Uset^{n}$,
$$
\Xmet(\tf{u}(z),\tf{u}(z'))=\left|\check{\psi}_n^{\theta}(z-z')\right|\;,
$$
where $|\cdot|$ is a norm on $\rset^{p'}$ and, for all $n\in\zsetpnz$,
$\check{\psi}_n^{\theta}$ is a linear mapping from
$\Zset=\rset^{p'*p+q'*(q-1)}$ to $\Xset=\rset^{p'}$ recursively
defined by setting, for all $w=\chunk w1{(p+q+1)}\in\Zset=(\rset^{p'})^p\times(\rset^{q'})^{q-1}$,
$\check{\psi}_n^{\theta}(w)=x_n$ with
\begin{align}
  \label{eq:pq-order-rec-equation-diff-lin}
  \begin{cases}
    u_j=w_{p+q+j} \;,&
    -q< j\leq -1\;, \\
    u_j=0 \;,&
    0\leq j\leq n-1\;, \\
    x_j=w_{p+j} \;,&
    -p< j\leq 0\;,\\
    x_j=\sum_{k=1}^pA_k(\theta)x_{j-k}+\sum_{k=1}^qB_k(\theta)u_{j-k}\;,&
    1\leq j\;.
  \end{cases}
\end{align}
In particular, we have, for all $j\geq q$,
$$
x_j=\sum_{k=1}^pA_k(\theta)x_{j-k}\;,
$$
and this equation is also true for $j=1,\dots,q-1$ if the last $q-1$,
$\rset^{q'}$-valued, component of $w$ are equal to zero. It is well
known that the Lipshitz norm of such iterative linear functions goes
to zero if and only if~\ref{item:vlodm-ident-cond-cond-dens-def} holds.

\noindent\textbf{Proof of
  Assertion~\ref{item:item:ident-cond-log-moment}.}   Take an arbitrary $\initmlex_1\in\Xset$. If $q>1$, take also an arbitrary  $\initmleu_1\in\Uset$ and set $\initmleu=(\initmleu_1,\dots,\initmleu_1) \in\Uset^{q-1}$.  Then, since
$\tilde \psi^\theta_u(x)$ is of the
form~\eqref{eq:tilde-psi:affine:vector:case}, there exists constants
$C_1,C_2>0$ only depending on $\theta$, $\initmlex_1$ and
$\initmleu_1$ such that, for all $u\in\Uset$,
$$
\Xmet\left(\initmlex_1,\tilde \psi^\theta_{(\initmleu,u)}(\initmlex)\right)
\leq C_1+C_2\,|u| \;.
$$
Assertion~\ref{item:item:ident-cond-log-moment} follows.

\noindent\textbf{Assertion~\ref{item:item:continuity-in-x-for-u-given}}
is obvious.

\noindent\textbf{Proof of Assertions~\ref{item:item:ident-nu-nondegenerate} and~\ref{item:item:ident-nu-nondegenerate-X}.} See
Remark~\ref{rem:ident-non-degenerate}~\ref{item:rem:ident-non-degenerate2}
in the case of a VLODM. The case of a VLODMX is similar.

\subsection{Proof of \autoref{lem:vlodm:angle-set}}
\label{sec:proof-lem:vlodm:angle-set}
Note that, by \autoref{lem:ergo:ass:1}~\ref{item:assum:bound:rho:gen},
in the vector linear case, the set $\mathrm{E}^\theta$ of
\autoref{def:E-sets} is well defined
under~\ref{item:vlodm-ident-cond-cond-dens-def}. We need the following
result whose proof is straightforward, and thus omitted.
\begin{lemma}
\label{lem:vlodm:Dtheta-Etheta}  Suppose that~\ref{item:vlodm-ident-cond-cond-dens-def} holds and let $\theta\in\Theta$.
Let $\ell^1(\zset,\Uset)$ denote the set of sequences in $\Uset^\zset$
that are absolutely summable.
For any
$u\in\ell^1(\zset,\Uset)$, there is a unique
$x\in\ell^\infty(\zset,\Xset)$ (the set of bounded sequences valued in $\Xset$)
such that
\begin{equation}
  \label{eq:pq-order-rec-equation-vlodm-one-param-allZ}
  x_j=\boldsymbol{\omega}(\theta)+\sum_{k=1}^p A_k(\theta)\,x_{j-k}+\sum_{k=1}^q B_k(\theta) \, u_{j-i}\;,
\qquad  j\in\zset\;.
\end{equation}
This unique solution is given by
$$
x_j=
\left(\mathrm{I}_{p'}-\sum_{k=1}^pA_k(\theta)\right)^{-1}\boldsymbol{\omega}(\theta)
+\int_{-\pi}^\pi\rme^{\rmi\lambda(j+p-q)}\mathbf{R}(\rme^{\rmi\lambda};\theta)
\hat u(\lambda)\;\rmd\lambda\;,\qquad j\in\zset\;,
$$
where $\mathbf{R}$ is defined by~(\ref{eq:RpDef}) and $\hat u$ denotes
the Fourier series of $u$ defined by
$$
\hat
u(\lambda)=\frac1{2\pi}\sum_{k\in\zset}u_k\rme^{-\rmi\lambda
  k}\,,\qquad\lambda\in\rset\;.
$$
Let $\mathrm{D}^\theta$ and $\mathrm{E}^\theta$ be as
in~(\ref{eq:definition:f:asalimit:with:domain}) and
\autoref{def:E-sets}. Then $\mathrm{D}^\theta$ contains
$\ell^1(\zset,\Uset)$ and, for any $u\in\ell^1(\zset,\Uset)$, defining
$x$ as the unique solution
of~(\ref{eq:pq-order-rec-equation-vlodm-one-param-allZ}) in
$\ell^\infty(\zset,\Xset)$, we have, for all $t\in\zset$,
$$
\left(\chunk x{(t-p+1)}{t},\chunk u{(t-q+1)}{(t-1)}\right)=\tF{\chunk u{(-\infty)}{(t-1)}}\in\mathrm{E}^\theta\;.
$$
\end{lemma}
We can now prove  \autoref{lem:vlodm:angle-set}.
\begin{proof}[Proof of  \autoref{lem:vlodm:angle-set}]
  \noindent\textbf{Step~1:
    Assertion~\ref{item:theta-well-chosen-E-again} implies Assertion~\ref{item:theta-well-chosen-E-CNS}.}  Let $\theta,\thv\in\Theta$
satisfying Assertion~\ref{item:theta-well-chosen-E-again} and let us show
that~(\ref{eq:vlodm:iden:omega-eq}) and~(\ref{eq:vlodm:iden:R-eq}) hold.
Take any $u\in\ell^1(\zset,\Uset)$ and $n\in\zsetp$. By \autoref{lem:vlodm:Dtheta-Etheta}, $u\in \mathrm{D}^\thv$ and we have
\begin{align*}
\tF[\thv]{\chunk u{(-\infty)}{0}}&=\tF[\thv]{\chunk u{(-n)}{0}}
                                     (\tF[\thv]{\chunk u{(-\infty)}{(-n-1)}})\\
  &=\tF[\theta]{\chunk u{(-n)}{0}}(\tF[\thv]{\chunk u{(-\infty)}{(-n-1)}})\; ,
\end{align*}
where the second equality follows from applying successively
Assertion~\ref{item:theta-well-chosen-E-again} with $u=u_{k}$ and
$z=\tF[\thv]{\chunk u{(-\infty)}{(k-1)}}$ for $k=-n,-n+1,\dots,0$.
On the other hand by definition of $\mathrm{Lip}_n^\theta$
in~(\ref{eq:Lip:constant:n:def}), we have, setting
$z^{\star}_{n}:=\tF[\thv]{\chunk u{(-\infty)}{(-n-1)}}$ and $z_n=\tF{\chunk u{(-\infty)}{(-n-1)}}$,
\begin{align*}
\Zmet\left(\tF[\theta]{\chunk u{(-\infty)}{0}},\tF[\theta]{\chunk
  u{(-n)}{(0)}}(z^{\star}_{n})\right)
  &=
\Zmet\left(\tF[\theta]{\chunk u{(-n)}{0}}\left(z_{n}\right),\tF[\theta]{\chunk
  u{(-n)}{(0)}}(z^{\star}_{n})\right)
\\
  &\leq
\mathrm{Lip}_n^\theta\;\;\Zmet\left(z_{n},z^{\star}_{n}\right)\;.
\end{align*}
Since $u\in\ell^1(\zset,\Uset)$, we have that $(z_n)$ and $(z^{\star}_n)$
are summable sequences (as a consequence of
\autoref{lem:vlodm:Dtheta-Etheta}) and so
$\Zmet\left(z_{n},z^{\star}_{n}\right)$ is bounded as $n\to\infty$.
Using  \autoref{lem:ergo:ass:1}~\ref{item:assum:bound:rho:gen} and we
conclude that the upper bound in the
last display converges to 0 as $n\to\infty$. By definition of $z^{\star}_n$
and the previous display, this gives that $\tF[\theta]{\chunk
  u{(-\infty)}{0}}= \tF[\thv]{\chunk u{(-\infty)}{0}}$.
Shifting the sequence $u$, we also have that
$\tF[\theta]{\chunk
  u{(-\infty)}{t}}= \tF[\thv]{\chunk u{(-\infty)}{t}}$ for all
$t\in\zset$ and by \autoref{lem:vlodm:Dtheta-Etheta}, this implies
that $\theta$ and $\thv$ share the same unique solution
$x\in\ell^\infty(\zset,\Xset)$ to the
equation~(\ref{eq:pq-order-rec-equation-vlodm-one-param-allZ}).
Using the explicit form of this solution in the same lemma, we get
that, for all  $u\in\ell^1(\zset,\Uset)$ and $\tau\in\zset$,
$$
\left(\mathrm{I}_{p'}-\sum_{k=1}^pA_k(\theta)\right)^{-1}\boldsymbol{\omega}(\theta)
+\alpha_\tau(u;\theta)=\left(\mathrm{I}_{p'}-\sum_{k=1}^pA_k(\thv)\right)^{-1}\boldsymbol{\omega}(\thv)
+\alpha_\tau(u;\thv)\;,
$$
where, for all $\theta\in\Theta$, $u\in\ell^1(\zset,\Uset)$ and $\tau\in\zset$, we set
$\displaystyle\alpha_\tau(u;\theta)=\int_{-\pi}^\pi\rme^{\rmi\lambda\,\tau}\mathbf{R}(\rme^{\rmi\lambda};\theta)
\hat u(\lambda)\;\rmd\lambda$. Since we assumed
$\{0\}\subsetneq\Uset$, we can successively take 
$u$ as the zero sequence ($u_k=0$ for all $k$, implying $\hat
u\equiv0$) or proportional to the impulse sequence  ($u_0\neq0$,
$u_k=0$ for all $k\neq0$, implying $\hat u\equiv u_0/(2\pi)$), the previous
display successively leads to~(\ref{eq:vlodm:iden:omega-eq}) and
$$
\int_{-\pi}^\pi\rme^{\rmi\lambda\,\tau}\;\mathbf{R}(\rme^{\rmi\lambda};\theta)
\;\rmd\lambda=\int_{-\pi}^\pi\rme^{\rmi\lambda\,\tau}\;\mathbf{R}(\rme^{\rmi\lambda};\thv)\;,\quad\text{for
  all}\quad \tau\in\zset\;,
$$
which implies~(\ref{eq:vlodm:iden:R-eq}).

\noindent\textbf{Step~2: Assertion~\ref{item:theta-well-chosen-E-CNS} implies Assertion~\ref{item:theta-well-chosen-E-again}.}
  Let $\theta,\thv\in\Theta$
satisfying~(\ref{eq:vlodm:iden:omega-eq})
and~(\ref{eq:vlodm:iden:R-eq}), and let us show
that  Assertion~\ref{item:theta-well-chosen-E-CNS} holds.
First take $u\in\ell^1(\zset,\Uset)$. By
\autoref{lem:vlodm:Dtheta-Etheta},~(\ref{eq:vlodm:iden:omega-eq})
and~(\ref{eq:vlodm:iden:R-eq}) imply that the recursive
equation~(\ref{eq:pq-order-rec-equation-vlodm-one-param-allZ}) and the
one with $\theta$ replaced by $\thv$  share
the same bounded solution. Moreover,
we have
$u\in\mathrm{D}^\thv\cap\mathrm{D}^\theta$ and since $\tF{\chunk u{(-\infty)}0}$
$\tF[\thv]{\chunk u{(-\infty)}0}$ are given by the same solution they
are equal. Hence we obtain that
$$
\tilde\Psi^\theta_u(z)=\tilde\Psi^{\thv}_u(z)\qquad\text{for all}\quad
(z,u)\in\mathrm{E}^\thv_1\times\Uset\;,
$$
where $\mathrm{E}^\thv_1=\set{\tF[\thv]{v}}{v\in\ell^1(\zsetn,\Uset)}$. To
get Assertion~\ref{item:theta-well-chosen-E-again}, since
$z\mapsto\tilde\Psi^{\theta'}_u(z)$ is continuous for
$\theta'=\theta,\thv$ and for any $u\in\rset^{q'}$,
it is now sufficient to
prove that $\mathrm{E}^\thv_1$ is dense in $\mathrm{E}^\thv$.
To this end, pick $z\in\mathrm{E}^\thv$. Then there exists
$u\in\mathrm{D}^\thv$
such that
\begin{equation}
  \label{eq:vlodm:density-of-E1-z-given}
z=\displaystyle  \tF[\theta]{u}\;.
\end{equation}
Define, for any $n\in\zsetp$, we introduce the truncated  sequence
$$
 v^{(n)}_k=
 \begin{cases}
u_k
&\text{ if $k\in\{-n,\dots,0\}$}\\
0&\text{ otherwise.}
 \end{cases}
$$
Then $v^{(n)}\in\ell^1(\zset,\Uset)$ and we have
$$
z_n:=\tF[\thv]{\chunk
  {v^{(n)}}{(-\infty)}{0}}\in\mathrm{E}_1^\thv\;.
$$
Moreover, we can write, denoting by $\chunk 0{(-\infty)}{0}$ the null
sequence in $\Uset^{\zsetn}$,
$$
z_n=\tF[\thv]{\chunk v{(-n)}{0}}\left(\tF[\thv]{\chunk
    0{(-\infty)}{0}}\right)
=\tF[\thv]{\chunk u{(-n)}{0}}\left(\tF[\thv]{\chunk 0{(-\infty)}{0}}\right)\;.
$$
By~(\ref{eq:definition:fF:asalimit:with:domain})
and~(\ref{eq:vlodm:density-of-E1-z-given}) we thus have
$z=\lim_{n\to\infty}z_n$, and since $z$ is arbitrary in
$\mathrm{E}^\thv$ we have shown that  $\mathrm{E}^\thv_1$ is dense in
$\mathrm{E}^\thv$ and the proof is concluded.
\end{proof}

\section*{Acknowledgments}
In the first version of this contribution, we neither investigated the
VLODM case nor the case with exogenous covariates.  We are grateful to
the two referees that reviewed the first submission, whose fruitful
and constructive comments motivated these demanding extensions.

\end{document}